\documentclass[a4paper,12pt]{amsart}
\usepackage{amsmath,amssymb}
\usepackage{amsthm}
\usepackage[abbrev]{amsrefs}
\usepackage{latexsym}
\usepackage{centernot}
\usepackage{comment}
\usepackage[pdftex]{hyperref}
\hypersetup{setpagesize=false, bookmarks=true, bookmarksdepth=tocdepth, bookmarksnumbered=true, colorlinks=true, linkcolor=blue, citecolor=blue, pdftitle={}, pdfsubject={}, pdfauthor={}, pdfkeywords={}}
\theoremstyle{plain}
  \newtheorem{thm}{Theorem}[section]
  \newtheorem{mthm}[thm]{Main Theorem}
  \newtheorem{lem}[thm]{Lemma}
  
  \newtheorem{cor}[thm]{Corollary}
  \newtheorem{prop}[thm]{Proposition}
  \newtheorem{claim}[thm]{Claim}

\theoremstyle{definition}
  \newtheorem{dfn}[thm]{Definition}
  
  \newtheorem{ex}[thm]{Example}

\theoremstyle{remark}
  \newtheorem{rem}[thm]{Remark}

\newcommand{\Lip}{\mathcal{L}{\it ip}}
\newcommand{\N}{\mathbb{N}}

\newcommand{\Q}{\mathbb{Q}}
\newcommand{\R}{\mathbb{R}}

\newcommand{\B}{\mathcal{B}}
\newcommand{\D}{\mathcal{D}}
\newcommand{\J}{\mathcal{J}}
\newcommand{\K}{\mathcal{K}}
\newcommand{\LL}{\mathcal{L}}

\renewcommand{\P}{\mathcal{P}}
\newcommand{\Ss}{\mathbb{S}}
\newcommand{\X}{\mathcal{X}}
\newcommand{\NN}{\frac{1}{N'}}
\newcommand{\1}{\mathbf{1}}
\newcommand{\al}{\alpha}
\newcommand{\be}{\beta}
\newcommand{\ga}{\gamma}
\newcommand{\de}{\delta}
\newcommand{\ep}{\varepsilon}
\newcommand{\ka}{\kappa}
\newcommand{\ph}{\varphi}
\newcommand{\la}{\lambda}
\newcommand{\si}{\sigma}
\newcommand{\om}{\omega}

\newcommand{\unu}{\underline{\nu}}
\newcommand{\urho}{\underline{\rho}}
\newcommand{\tB}{\tilde{B}}
\newcommand{\tU}{\tilde{U}}
\newcommand{\tX}{\tilde{X}}
\newcommand{\jkmn}{\xi_{jk}^{mn}}
\newcommand{\kjmn}{\xi_{kj}^{mn}}
\newcommand{\jkm}{\xi_{jk}^m}

\newcommand{\pmG}{{\rm pmG}}
\newcommand{\pmGH}{{\rm pmGH}}
\newcommand{\del}{\partial}

\newcommand{\wto}{\rightharpoonup}
\newcommand{\concto}{\xrightarrow{{\rm conc}}}
\newcommand{\Boxto}{\xrightarrow{\Box}}
\newcommand{\nll}{\centernot{\ll}}

\DeclareMathOperator{\Opt}{Opt}
\DeclareMathOperator{\Cpl}{Cpl}
\DeclareMathOperator{\Ric}{Ric}
\DeclareMathOperator{\vol}{vol}
\DeclareMathOperator{\dKF}{{\it d}_{{\rm KF}}}
\DeclareMathOperator{\dKFH}{{\it d}_{\it H}^{{\rm KF}}}
\DeclareMathOperator{\dconc}{{\it d}_{{\rm conc}}}
\DeclareMathOperator{\Hess}{Hess}
\DeclareMathOperator{\proj}{proj}
\DeclareMathOperator{\dist}{dist}

\DeclareMathOperator{\diam}{diam}
\DeclareMathOperator{\supp}{supp}
\DeclareMathOperator{\sgn}{sgn}
\DeclareMathOperator{\grad}{grad}
\DeclareMathOperator{\CD}{CD}

\DeclareMathOperator{\Sep}{Sep}
\DeclareMathOperator{\Obsdiam}{ObsDiam}

\begin{document}

\title[Stability of CD condition for negative dimensions]{Stability of curvature-dimension condition for negative dimensions under concentration topology}

\author{Shun Oshima}
\address{Mathematical Institute, Tohoku University,
Sendai 980-8578, Japan}
\email{shun.oshima.s3@dc.tohoku.ac.jp}

\begin{abstract}
In this paper, we prove the stability of metric measure spaces satisfying the curvature-dimension condition for negative dimensions under a concentration topology. This result is an analogue of the result by Funano--Shioya with respect to the dimension parameter.
\end{abstract}

\date{\today}


\maketitle
\section{Introduction}
A metric measure space is a metric space with the structure of a measure space, defined as a generalization of a Riemannian manifold with a Riemannian distance and a volume measure determined by a Riemannian metric. In particular, it has been studied whether the properties and structure of Riemannian manifolds hold in metric measure spaces. There are many definitions of metric measure spaces, but this paper deals specifically with metric measure spaces with probability measures, which we call mm-spaces for short.

There are two main concepts in this paper. The first is the curvature-dimension condition ($\CD(K,N)$ condition, $K\in \R$, $N\in (1,\infty]$), which is a condition that the weighted Ricci curvature $\Ric_N$ of an $n$-dimensional weighted Riemannian manifold defined as
\begin{align*}
\Ric_N(v, v):=\Ric_g(v,v)+\Hess f(v,v)-\frac{\langle \nabla f(x),v\rangle}{N-n}\qquad(v\in T_xM)
\end{align*}
satisfies
\begin{equation*}
\Ric_N\ge Kg.
\end{equation*}
It is well known that this condition is equivalent to ``Ricci curvature greater than or equal to $K$ and dimension of manifold less than or equal to $N$" for Riemannian manifolds. The extension of this condition to the metric measure spaces is given independently by Lott--Villani \cite{Lott-Villani} and Sturm \cite{Sturm1}, \cite{Sturm2}. A metric measure space satisfying this condition is called a $\CD(K,N)$ space. It is known that $\CD(K,N)$ spaces inherit many properties of Riemannian manifolds with Ricci curvature more than $K$ and dimension less than $N$. Then, by considering the weighted Ricci curvature as $N<0$, the $\CD(K,N)$ condition for $N<0$ is defined, which is weaker than the $\CD(K,\infty)$ condition, giving an even wider class of spaces. 
Furthermore, The $\CD(K,N)$ condition for $N<0$ was extended to a metric measure space by Ohta \cite{Ohta}, and this allows us to consider a metric measure space satisfying the $\CD(K,N)$ condition for $N<0$.

The second is the convergence of metric measure spaces, in particular, the $\Box$-convergence and concentration introduced by Gromov \cite{Gromov}, which is the convergence of metric measure spaces with probability measure (called mm-spaces). There are various types of convergence of metric measure spaces known today, some examples are
\begin{enumerate}
\item pointed measured Gromov-Hausdorff convergence ($\pmGH$ convergence)
\item pointed measured Gromov convergence ($\pmG$ convergence)
\item Sturm's $\mathbb{D}$-convergence
\item $\Box$-convergence
\item Concentration.
\end{enumerate}
Whether these convergences can be considered depends on the definition of the metric measure spaces, but, in mm-spaces, all convergences can be considered, and (5) is known to be the weakest convergence. Funano--Shioya \cite{Funano-Shioya} and Kazukawa--Ozawa--Suzuki \cite{Kazukawa-Ozawa-Suzuki} obtained the following results on the question whether the curvature dimension condition is stable for the concentration topology which is the weakest convergence.
\begin{thm}[{\cite[Theorem 1.2]{Funano-Shioya}}, {\cite[Theorem 1.1]{Kazukawa-Ozawa-Suzuki}}]
\label{concCDKinfty}
If a sequence of mm-spaces $\{X_n\}_{n\in \N}$ satisfying the $\CD(K,\infty)$ condition for some $K\in \R$ concentrates to some mm-space $Y$, then $Y$ is also a $\CD(K,\infty)$ space.
\end{thm}
This shows that the $\CD(K,\infty)$ condition is stable in the concentration topology. A result on the question whether the $\CD(K,N)$ condition for $N<0$ is stable for convergence of metric measure spaces is also given by Magnabosco--Rigoni--Sosa \cite{Magnabosco-Rigoni-Sosa}.
\begin{thm}[{\cite[Theorem 4.1]{Magnabosco-Rigoni-Sosa}}]
\label{pmGKN}
Suppose that a sequence $\{(X_n,d_{X_n},\mu_{X_n})\}_{n\in \N}$ of metric measure spaces satisfying the $\CD(K,N)$ condition for some $K\in \R$ and $N<0$ converges to some metric measure space $(Y,d_Y,\mu_Y)$ in the $\pmG$ sense. If the following assumption  
\begin{equation*}
\limsup_{n\to \infty}\diam X_n<\frac{\pi}{\sqrt{-K}}
\end{equation*}
holds when $K<0$, then $Y$ is also a $\CD(K,N)$ space.
\end{thm}
Note that Magnabosco et al. \cite{Magnabosco-Rigoni-Sosa} proved the above theorem for a broader definition of the metric measure spaces than that in this paper, and Theorem \ref{pmGKN} is an application of their theorem to the definition in this paper.
The main theorem in this paper is that the $\CD(K,N)$ condition for $N<0$ is stable under the concentration topology weaker than the topology of $\pmG$ convergence.
\begin{mthm}
\label{mainthm1}
Suppose that a sequence $\{X_n\}_{n\in \N}$ of mm-spaces satisfying the $\CD(K,N)$ condition for some $K\in \R$ and $N<0$ concentrates to some mm-space $Y$. If the following assumption
\begin{equation}
\label{diamassume}
\limsup_{n\to \infty}\diam X_n<\frac{\pi}{\sqrt{-K}}
\end{equation}
holds when $K<0$, then $Y$ is also a $\CD(K,N)$ space.
\end{mthm}
\begin{rem}
\label{rem}
If $K\ge 0$, then Main Theorem \ref{mainthm1} with the $\CD(K,N)$ condition replaced by the $\CD^*(K,N)$ condition can be proved similarly.
\end{rem}
It can also be shown that it is necessary to assume (\ref{diamassume}) for $K<0$.
\begin{mthm}
\label{mainthm2}
Let $K,N$ be negative numbers and let $D\ge\pi\sqrt{(N-1)/K}$. Then, there exist 
a sequence $\{f_n\}_{n\in \N}$ of smooth functions on the circle $\Ss^1$ 
and an mm-space $(Y, d_Y, \mu_Y)$ consisting of two points with $\diam Y=D$ such that a sequence of mm-spaces $\{(X_n,d_{X_n},\mu_{X_n})\}_{n\in \N}$ defined as 
\begin{equation*}
(X_n,d_{X_n},\mu_{X_n}):=\left(\Ss^1, d_g, e^{-f_n}\vol_g\right),
\end{equation*}
where $g$ is a metric the canonical metric on $\Ss^1$ scaled to $\diam \Ss^1=D$, satisfies the $\CD(K,N)$ condition and $\Box$-converges to $(Y, d_Y, \mu_Y)$.
\end{mthm}
In particular, since an mm-space consisting of two points is not a $\CD(K,N)$ space, we can say that $X_n$ and $Y$ constructed in Main Theorem \ref{mainthm2} are counterexamples of Main Theorem \ref{mainthm1} without assuming (\ref{diamassume}). This result is clearly different from the case $1<N\le \infty$ where the concentration limit of $\CD(K,N)$ spaces is also a $\CD(K,N)$ space for any $K\in \R$ without the assumption (\ref{diamassume}) as in Theorem \ref{concCDKinfty}.

Finally, we describe the structure of this paper. In Section 2, we state some basic propositions on measure spaces and Wasserstein spaces.  In Section 3, we state definitions, examples, and properties of metric measure spaces, curvature-dimension conditions, and concentration topology, and in Section 4, we give some results on the estimates of observable diameter given in Section 3. In Section 5, we state the necessary lemmas for the proof of the Main theorem \ref{mainthm1}, such as the properties of R\'{e}nyi entropy. In Sections 6 and 7, we give the proofs of the Main theorems \ref{mainthm1} and \ref{mainthm2}, respectively. In Section 8 we present some results for metric measure spaces that satisfy the $\CD(K,N)$ condition for some $N<0$ and are not mm-spaces.

\section*{Notations}
We list the notations we will use throughout this paper.
\begin{itemize}
\item For $x,y\in \R$, we write $x\lor y:=\max\{x,y\}$ and $x\land y:=\min\{x,y\}$.
\item $\B_X$ denotes the set of all Borel subsets of a topological space $X$.
\item $\P(X)$ denotes the set of all Borel probability measures on $X$.
\item $C_b(X)$ denotes the set of all bounded continuous functions on $X$.
\item $\Lip(X)$ (resp. $\Lip_L(X)$) denotes the set of all Lipschitz (resp. $L$\,-Lipschitz) functions on $X$.
\item For $i=0,1$, $\proj_i:X\times X\to X$ denotes the projection defined by $\proj_i(x_0,x_1):=x_i$
\item For $\mu,\nu\in \P(X)$, we write $\nu\ll\mu$ to mean that $\nu$ is absolutely continuous with respect to $\mu$, and $\nu\nll\mu$ if it is not.
\item $\1_A$ denotes the characteristic function of a subset $A\subset X$.
\item For maps $p,q:X\to Y$, $p\times q:X\times X\to Y\times Y$ denotes the product of $p$ and $q$ defined by $(p\times q)(x_0,x_1):=(p(x_0),q(x_1))$.
\item $B_X(x,r)$ denotes the open ball in a metric space $(X,d_X)$ for the center $x\in X$ and the radius $r>0$.
\item $N_\ep(A)$ denotes the $\ep$-neighbourhood of a subset $A$ of a metric space $(X,d_X)$, namely
\begin{equation*}
N_\ep(A):=\bigcup_{a\in A}B_X(a,\ep)
\end{equation*}
\end{itemize}
\section{Preliminaries}
In this section, we give some basic definitions and propositions about metric space, probability measure space, and Wasserstein space, e.g. \cite{Villani}.
\begin{prop}
\label{closurextension}
Let $(X,d_X)$ be a metric space and let $(Y,d_Y)$ be a complete metric space. If a map $f:A\to Y$ on a subset $A$ of $X$ is $L$\,-Lipschitz for some $L>0$, then there exists a unique $L$\,-Lipschitz extension of $f$ on $\overline{A}$ which is the closure of $A$.
\end{prop}
This unique extension of $f$ is also denoted by $f$ as well.
\begin{dfn}[Weakly convergence]
Let $X$ be a topological space. We say that \emph{$\{\mu_n\}_{n\in \N}\subset \P(X)$ converges to $\mu\in \P(X)$ weakly} if for any $f\in C_b(X)$,
\begin{equation*}
\lim_{n\to \infty}\int_Xfd\mu_n=\int_Xfd\mu
\end{equation*}
holds and write by $\mu_n\wto \mu$.
\end{dfn}
There are several conditions that are equivalent to this definition of weak convergence, and they are known as the Portmanteau theorem. It is also known that when $X$ is a separable metric space, the topology on $\P(X)$, which is determined by weak convergence, can be metrized by the \emph{Prokhorov distance $d_P$}. Furthermore, when $X$ is complete, Prokhorov's theorem says that $\K\subset \P(X)$ is a relatively compact if and only if $\K$ is \emph{tight}, i.e., for any $\ep>0$, there exists a compact subset $K\subset X$ such that  
\begin{equation*}
\sup_{\mu\in \K}\mu(X\setminus K)<\ep.
\end{equation*}
In the following, $X$ is a complete and separable metric space.
\begin{dfn}[Coupling]
Let $\mu_0, \mu_1\in \P(X)$. We say that $\pi \in \P(X\times X)$ is a \emph{coupling of $\mu_0$ and $\mu_1$} if for any $i=0,1$, $(\proj_i)_*\pi$ called a push-forward of $\pi$ by $\proj_i$ equals $\mu_i$. The set of all couplings of $\mu_0$ and $\mu_1$ is denoted by $\Cpl(\mu_0,\mu_1)$.
\end{dfn}
\begin{lem}
\label{Cpltight}
Let $\{\mu_n\}_{n\in \N}, \{\nu_n\}_{n\in \N}\subset \P(X)$, $\mu,\nu\in \P(X)$ and $\pi_n\in \Cpl(\mu_n,\nu_n)$. If $\mu_n$, $\nu_n$ weakly converge to $\mu$, $\nu$ respectively, then the following holds.
\begin{itemize}
\item $\{\pi_n\}_{n\in \N}$ is tight.
\item If $\pi_n$ converges to some $\pi\in \P(X\times X)$ weakly, then $\pi$ is a coupling of $\mu$ and $\nu$.
\end{itemize}
\end{lem}
\begin{dfn}[Wasserstein space]
We define
\begin{equation*}
\P_2(X):=\left\{\mu\in \P(X)\left|\int_Xd_X(x,x_0)^2d\mu(x)<\infty\text{ for some $x_0\in X$}\right\}\right.
\end{equation*}
and $W_2:\P_2(X)\times \P_2(X)\to [0,\infty)$ defined by 
\begin{equation}
\label{Wpdef}
W_2(\mu, \nu):=\inf_{\pi\in \Cpl(\mu,\nu)}\left(\int_{X\times X}d_X(x,y)^2d\pi(x,y)\right)^{\frac{1}{2}}
\end{equation}
is a metric on $\P_2(X)$. We say that the metric space $(\P_2(X),W_2)$ is the \emph{Wasserstein space}. $\pi \in \Cpl(\mu,\nu)$ attaining the infimum of the right-hand side of (\ref{Wpdef}) is called an \emph{optimal coupling of $\mu$ and $\nu$}. We denote the set of all optimal couplings of $\mu$ and $\nu$ by $\Opt(\mu,\nu)$.
\end{dfn}
It is known that $(\P_2(X),W_2)$ is a complete and separable metric space.
\begin{prop}
\label{Wpconv}
Let $\{\mu_n\}_{n\in \N}\subset \P_2(X)$ and $\mu\in \P_2(X)$. The following are equivalent.
\begin{enumerate}
\item $\mu_n$ $W_2$-converges to $\mu$ i.e. $\displaystyle \lim_{n\to \infty}W_2(\mu_n,\mu)=0$
\item $\mu_n$ converges to $\mu$ weakly and 
\begin{equation*}
\lim_{R\to \infty}\limsup_{n\to \infty}\int_{X\setminus B_X(x_0, R)}d_X(x,x_0)^2d\mu_n(x)=0
\end{equation*}
holds for some $x_0\in X$.
\end{enumerate}
\end{prop}
\begin{lem}
\label{Wplsc}
Let $\{\mu_n\}_{n\in \N}, \{\nu_n\}_{n\in \N}\subset \P_2(X)$.  If $\mu_n$, $\nu_n$ weakly converge to $\mu, \nu\in \P_2(X)$ respectively, then we have
\begin{equation*}
W_2(\mu,\nu)\le \liminf_{n\to \infty}W_2(\mu_n,\nu_n).
\end{equation*}
In addition, suppose that $\mu_n$ and $\nu_n$ $W_2$-converge to $\mu$ and $\nu$ respectively. If $\pi_n\in \Opt(\mu_n, \nu_n)$ converges to some $\pi\in \P(X\times X)$ weakly, then $\pi$ is an optimal coupling of $\mu$ and $\nu$.
\end{lem}
\begin{dfn}
Let $\mu \in \P(X)$. For $B\in \B_X$ with $\mu(B)>0$, we define $\mu_B\in \P(X)$ by
\begin{equation*}
\mu_B(A)=(\mu)_B(A):=\frac{\mu(B\cap A)}{\mu(B)}
\end{equation*}
for $A\in \B_X$.
\end{dfn}
\begin{lem}
\label{wconvint}
Let $\{\nu_0^n\}_{n\in \N}, \{\nu_1^n\}_{n\in \N}\subset \P(X)$, $\nu_0, \nu_1 \in \P(X)$, $\pi^n\in \Cpl(\nu_0^n,\nu_1^n)$, $\ph\in C_b(X\times X)$ and let $f_i\in L^1(X,\nu_i)$ $(i=0,1)$. Suppose that $\nu_0^n$, $\nu_1^n$ and $\pi^n$ weakly converge to $\nu_0$, $\nu_1$ and some $\pi\in \Cpl(\nu_0,\nu_1)$ respectively. If either of the following three conditions
\begin{enumerate}
\item[(a)] for any $i=0,1$, $f_i$ is a Borel simple function such that for any $y\in \R$, $\nu_i(\del f_i^{-1}(y))=0$ where $\del f_i^{-1}(y)$ is the boundary of $f_i^{-1}(y)$
\item[(b)] there exists $C>0$ such that $\nu_i^n\le C\nu_i$ for $i=0,1$ 
\item[(c)] there exists $\mu\in \P(X)$ such that $f_i\in \L^{\infty}(X,\mu)$, $\nu_i^n=\rho_i^n\mu$, $\nu_i=\rho_i\mu$ and $\rho_i^n$ converges to $\rho_i$ in $L^1(X,\mu)$ for $i=0,1$
\end{enumerate}
holds, then we have
\begin{equation}
\label{intusc}
\limsup_{n\to \infty}\int_{X\times X}\ph(x_0,x_1)f_i(x_i)d\pi^n(x_0,x_1)\le\int_{X\times X}\ph(x_0,x_1)f_i(x_i)d\pi(x_0,x_1)
\end{equation}
for any $i=0,1$.
\end{lem}
\begin{proof}
If (a) holds, for any $i=0,1$, we put 
\begin{equation*}
f_i=\sum_{j\in \J}a_{i,j}\1_{B_{i,j}}
\end{equation*}
where a finite set $\J$, $\{a_{i,j}\}_{j\in\J}\subset \R$ and mutually disjoint Borel subsets $\{B_{i,j}\}_{j\in \J}$ satisfy $\nu_i(\del B_{i,j})=0$ for any $j\in \J$. Then, we have 
\begin{equation*}
\lim_{n\to \infty}\int_{\tB_{i,j}}\ph d\pi^n=\int_{\tB_{i,j}}\ph d\pi
\end{equation*}
where $\tB_{i,j}:=\proj_i^{-1}(B_{i,j})$. Indeed, it is obvious if $\pi(\tB_{i,j})=0$. If $\pi(\tB_{i,j})>0$, $(\pi^n)_{\tB_{i,j}}$ converges to $\pi_{\tB_{i,j}}$ weakly by $\pi(\del \tB_{i,j})=\nu_i(\del B_{i,j})=0$. Hence, 
\begin{equation*}
\int_{\tB_{i,j}}\ph d\pi^n=\pi^n(\tB_{i,j})\int_{X\times X}\ph d(\pi^n)_{\tB_{i,j}}\to \pi(\tB_{i,j})\int_{X\times X}\ph d\pi_{\tB_{i,j}}=\int_{\tB_{i,j}}\ph d\pi
\end{equation*}
holds. Thus, we get
\begin{align*}
&\lim_{n\to \infty}\int_{X\times X}\ph(x_0,x_1)f_i(x_i)d\pi^n(x_0,x_1)\\
=&\lim_{n\to \infty}\sum_{j\in \J}a_{i,j}\int_{\tB_{i,j}}\ph d\pi^n
=\int_{X\times X}\ph(x_0,x_1)f_i(x_i)d\pi(x_0,x_1).
\end{align*}
In particular, $(\ref{intusc})$ holds.

Next, we take any $\ep>0$. Then, there exists $f_i^{\ep}\in C_b(X)$ such that 
\begin{equation*}
\int_X|f_i-f_i^{\ep}|d\nu_i<\ep.
\end{equation*}
If (b) holds, we get
\begin{align*}
&\int_{X\times X}\ph(x_0,x_1)f_i(x_i)d\pi^n(x_0,x_1)\\
\le&\int_{X\times X}\ph(x_0,x_1)f_i^{\ep}(x_i)d\pi^n(x_0,x_1)\\
&+\int_{X\times X}|\ph(x_0,x_1)|\cdot|f_i(x_i)-f_i^\ep(x_i)|d\pi^n(x_0,x_1)\\
\le&\int_{X\times X}\ph(x_0,x_1)f_i^{\ep}(x_i)d\pi^n(x_0,x_1)+\|\ph\|_{\infty}\int_X|f_i-f_i^\ep|d\nu_i^n\\
\le&\int_{X\times X}\ph(x_0,x_1)f_i^{\ep}(x_i)d\pi^n(x_0,x_1)+C\cdot\|\ph\|_{\infty}\cdot \ep
\end{align*}
for any $i=0,1$. Since $\ph(x_0,x_1)f_i^{\ep}(x_i)$ is bounded continuous on $X\times X$, we have
\begin{align*}
&\limsup_{n\to\infty}\int_{X\times X}\ph(x_0,x_1)f_i(x_i)d\pi^n(x_0,x_1)\\
\le&\int_{X\times X}\ph(x_0,x_1)f_i^{\ep}(x_i)d\pi(x_0,x_1)+C\cdot\|\ph\|_{\infty}\cdot \ep\\
\le&\int_{X\times X}\ph(x_0,x_1)f_i(x_i)d\pi(x_0,x_1)+(C+1)\|\ph\|_{\infty}\cdot \ep.
\end{align*}
Thus, as $\ep\to 0$, we obtain (\ref{intusc}).

If (c) holds, we get
\begin{align*}
&\int_{X\times X}\ph(x_0,x_1)f_i(x_i)d\pi^n(x_0,x_1)\\
\le&\int_{X\times X}\ph(x_0,x_1)f_i^{\ep}(x_i)d\pi^n(x_0,x_1)+\|\ph\|_{\infty}\int_X|f_i-f_i^\ep|d\nu_i^n\\
\le&\int_{X\times X}\ph(x_0,x_1)f_i^{\ep}(x_i)d\pi^n(x_0,x_1)\\
&+\|\ph\|_{\infty}\left(\int_X|f_i-f_i^\ep|\cdot |\rho_i^n-\rho_i|d\mu+\int_X|f_i-f_i^\ep|d\nu_i\right)\\
\le&\int_{X\times X}\ph(x_0,x_1)f_i^{\ep}(x_i)d\pi^n(x_0,x_1)\\
&+\|\ph\|_{\infty}\left(\left(\|f_i\|_{\infty}+\|f_i^\ep\|_{\infty}\right)\int_X|\rho_i^n-\rho_i|d\mu+\ep\right)
\end{align*}
for any $i=0,1$. Since $\ph(x_0,x_1)f_i^{\ep}(x_i)$ is bounded continuous on $X\times X$, we have
\begin{align*}
&\limsup_{n\to\infty}\int_{X\times X}\ph(x_0,x_1)f_i(x_i)d\pi^n(x_0,x_1)\\
\le&\int_{X\times X}\ph(x_0,x_1)f_i^{\ep}(x_i)d\pi(x_0,x_1)+\|\ph\|_{\infty}\cdot \ep\\
\le&\int_{X\times X}\ph(x_0,x_1)f_i(x_i)d\pi(x_0,x_1)+2\|\ph\|_{\infty}\cdot \ep.
\end{align*}
Thus, as $\ep\to 0$, we obtain (\ref{intusc}).
\end{proof}

\section{Mm-spaces and curvature-dimension conditions}
In this section, we give the definition and basic properties of mm-space and the curvature-dimension condition.
\begin{dfn}[mm-space]
A triple $(X,d_X,\mu_X)$ is called an \emph{mm-space} if $(X,d_X)$ is a complete separable metric space and $\mu_X$ is a Borel probability measure on $X$.
\end{dfn}
\begin{dfn}[mm-isomorphism]
Two mm-spaces $(X,d_X,\mu_X)$, $(Y,d_Y,\mu_Y)$ are said to be \emph{mm-isomorphic} if there exists an isometry $f:\supp \mu_X\to \supp \mu_Y$ such that $f_*\mu_X$ equals $\mu_Y$. Such an $f$ is called an \emph{mm-isomorphism}.
\end{dfn}
Denote by $\X$ the set of mm-isomorphism classes of mm-spaces. 
\begin{dfn}[Parameter]
Let $I:=[0,1)$ and let $X$ be an mm-space. A Borel map $\ph:I\to X$ is called a \emph{parameter of $X$} if $\ph$ satisfies $\ph_*\LL^1=\mu_X$, where $\LL^1$ is the Lebesgue measure on $I$.
\end{dfn}
\begin{prop}[{\cite[Lemma 4.2]{ShioyaMMG}}]
Any mm-space has a parameter. 
\end{prop}
\begin{dfn}[Box distance between mm-spaces]
The \emph{box distance} $\Box(X,Y)$ between two mm-spaces $X$ and $Y$ is defined by the infimum of $\ep>0$ satisfying that there exist a Borel set $\tilde{I}\subset I$ and parameters $\ph:I\to X$ and $\psi:I\to Y$ of $X$ and $Y$ such that
\begin{enumerate}
\item $\LL^1(\tilde{I})\ge1-\ep$
\item $|d_X(\ph(s),\ph(t))-d_Y(\psi(s),\psi(t))|\le\ep$\qquad for any $s,t\in \tilde{I}$.
\end{enumerate}
\end{dfn}
It is known that $\Box$ is a metric on $\X$ and $(\X,\Box)$ is a complete metric space \cite[Theorem 4.10, Theorem 4.14]{ShioyaMMG}. If a sequence of mm-spaces $\{X_n\}_{n\in \N}$ converges to some mm-space $Y$ with respect to $\Box$, we say that $\{X_n\}_{n\in \N}$  \emph{$\Box$-converges to $Y$} and denote by $X_n\Boxto Y$.
\begin{prop}[{\cite[Proposition 4.12]{ShioyaMMG}}]
\label{Boxwconv}
Let $X$ be a complete separable metric space. For any two Borel probability measures $\mu$ and $\nu$ on $X$, we have
\begin{equation*}
\frac{1}{2}\Box((X,\mu),(X,\nu))\le \Box((2^{-1}X,\mu),(2^{-1}X,\nu))\le d_P(\mu,\nu).
\end{equation*}
In particular, if a sequence $\{\mu_n\}_{n\in \N}\subset \P(X)$ converges to some $\mu\in \P(X)$ weakly, then we have $(X,\mu_n)\Boxto (X,\mu)$.
\end{prop}
\begin{dfn}[Observable diameter]
Let $X$ be an mm-space and let $\al\ge 0$. We define the \emph{partial diameter $\diam(X;\al)$ of $X$} defined by
\begin{equation*}
\diam(X;\al)=\diam(\mu_X;\al):=\inf\{\diam A\mid A\in \B_X \text{ with } \mu_X(A)\ge \al\},
\end{equation*}
For $\ka\ge 0$, we define 
\begin{align*}
\Obsdiam(X;-\ka):=&\sup\{\diam(f_*\mu_X;1-\ka)\mid f\in \Lip_1(X)\}\\
\Obsdiam(X):=&\inf_{\ka>0}(\ka\lor \Obsdiam(X;-\ka)).
\end{align*}
We call $\Obsdiam(X;-\ka)$ (resp. $\Obsdiam(X)$) the \emph{$\ka$-observable diameter of $X$} (resp. \emph{observable diameter of $X$}).
\end{dfn}
\begin{rem}
\label{Obsdecrea}
The $\ka$-observable diameter is nonincreasing for $\ka$, and we usually consider only when $\ka<1$ by $\Obsdiam(X;-\ka)=0$ if $\ka\ge 1$.
\end{rem}
\begin{dfn}[L\'{e}vy family]
A sequence of mm-spaces $\{X_n\}_{n\in \N}$ is called a \emph{L\'{e}vy family} if
\begin{equation*}
\lim_{n\to \infty}\Obsdiam(X_n)=0
\end{equation*}
holds, or equivalently
\begin{equation*}
\lim_{n\to \infty}\Obsdiam(X_n;-\ka)=0
\end{equation*}
holds for any $\ka>0$.
\end{dfn}
\begin{dfn}[Separation distance]
Let $X$ be an mm-space and let $N\in \N$, $\ka_0, \ka_1,\ldots,\ka_N>0$. We define the \emph{separation distance $\Sep(X;\ka_0, \ka_1,\ldots,\ka_N)$ of $X$} as the supremum of $\min_{i\ne j}\dist(A_i,A_j)$, where $A_0, A_1,\ldots,A_N$ run over all Borel subsets satisfying $\mu_X(A_i)\ge \ka_i$ for $i=0,1,\ldots,N$.
\end{dfn}
\begin{prop}[{\cite[Proposition 2.26]{ShioyaMMG}}]
\label{ObsSep}
Let $X$ be an mm-spaces. For any $\ka>\ka'>0$, we have 
\begin{itemize}
\item $\Obsdiam(X;-2\ka)\le \Sep(X;\ka, \ka)$
\item $\Sep(X;\ka, \ka)\le \Obsdiam(X;-\ka')$.
\end{itemize}
\end{prop}
\begin{dfn}[Ky Fan metric]
The \emph{Ky Fan distance} $\dKF(f,g)$ between two measurable functions $f,g:\Omega \to \R$ on a probability measure space $(\Omega, \mu)$ is defined by 
\begin{equation*}
\dKF(f,g):=\inf\left\{\ep\ge0\mid \mu(\{x\in \Omega\mid |f(x)-g(x)|>\ep\})\le \ep\right\}.
\end{equation*}
This distance function $\dKF$ is called the \emph{Ky Fan metric}. 
\end{dfn}
\begin{dfn}[Observable distance]
The \emph{observable distance} $\dconc(X,Y)$ between two mm-spaces $X$ and $Y$ is defined as 
\begin{equation*}
\dconc(X,Y):=\inf_{\ph, \psi}\dKFH\left(\ph^*\Lip_1(X),\psi^*\Lip_1(Y)\right)
\end{equation*}
where $\ph$ and $\psi$ run over all parameters of $X$ and $Y$, $\ph^*\Lip_1(X):=\{f\circ\ph\mid f\in \Lip_1(X)\}$ which is a subset of the set of measurable functions on $(I, \LL^1)$ and $\dKFH$ is the Hausdorff distance with respect to the Ky Fan metric.
\end{dfn}
It is also known that $\dconc$ is a metric on $\X$ ({\cite[Theorem 5.16]{ShioyaMMG}}). If a sequence of mm-spaces $\{X_n\}_{n\in \N}$ converges to some mm-space $Y$ with respect to $\dconc$, we say that \emph{$\{X_n\}_{n\in \N}$ concentrates to $Y$} and denote by $X_n\concto Y$.
\begin{prop}[{\cite[Proposition 5.5]{ShioyaMMG}}]
\label{Boxconc}
Any $X,Y\in \X$ satisfy $\dconc(X,Y)\le \Box(X,Y)$. In particular, if a sequence of mm-spaces $\{X_n\}_{n\in \N}$ $\Box$-converges to some mm-space $Y$, then $\{X_n\}_{n\in \N}$ concentrates to $Y$.
\end{prop}
A topology on $\X$ induced by $\dconc$ is called the \emph{concentration topology}.
\begin{prop}[{\cite[Proposition 5.7, Corollary 5.8]{ShioyaMMG}}]
Let $*$ be the mm-space consisting of one point, i.e. $*:=(\{*\}, d_*, \de_*)$. We have
\begin{equation*}
\dconc(X,*)\le \Obsdiam(X)\le 2\dconc(X,*)
\end{equation*}
for any $X\in \X$. In particular, for a sequence of mm-spaces $\{X_n\}_{n\in \N}$, $X_n \concto *$ is equivalent to $\{X_n\}_{n\in \N}$ being a L\'{e}vy family.
\end{prop}
\begin{prop}[{\cite[Corollary 5.35, Proposition 9.31]{ShioyaMMG}}]
\label{OY14}
Suppose that a sequence of mm-spaces $\{X_n\}_{n\in \N}$ conentrates to some mm-space $Y$. Then, for any $n\in \N$, there exist a Borel map $p_n:X_n\to Y$, a compact set $\tX_n\subset X_n$ and $\ep_n>0$ such that
\begin{enumerate}
\item $\dKFH(\Lip_1(X_n),p_n^*\Lip_1(Y))\le \ep_n$ and $\ep_n\to 0$
\item $(p_n)_*\mu_{X_n}\wto \mu_Y$
\item $d_Y(p_n(x),p_n(x'))\le d_{X_n}(x,x')+\ep_n$\qquad for any $x,x'\in \tX_n$
\item $\mu_{X_n}(\tX_n)\ge 1-\ep_n$
\item $\displaystyle \limsup_{n\to \infty}\sup_{x\in X_n\setminus \tX_n}d_Y(p_n(x),y) <\infty$\qquad for any $y\in Y$.
\end{enumerate}
Conversely, if there exist $p_n$ and $\ep_n$ satisfying $(1)$ and $(2)$, then $X_n$ concentrates to $Y$.
\end{prop}
\begin{prop}[{\cite[Proposition 2.6]{Kazukawa-Ozawa-Suzuki}}]
\label{diamlsc}
Suppose that a sequence of mm-spaces $\{X_n\}_{n\in \N}$ conentrates to some mm-space $Y$. Then we have
\begin{equation*}
\diam Y\le \liminf_{n\to \infty}\diam X_n.
\end{equation*}
\end{prop}
\begin{dfn}[$k$-th eigenvalue]
Let $M$ be a compact Riemannian manifold and let $k\in \N\cup\{0\}$. We define the \emph{$k$-th eigenvalue $\la_k(M)$ of the Laplacian of $M$} by
\begin{equation*}
\la_k(M):=\inf_L\sup_{u\in L\setminus \{0\}}R(u),
\end{equation*}
where the measure of $M$ is the normalized volume measure, $L$ runs over all $(k+1)$-dimensional subspaces of $L^2(M)\cap \Lip(M)$, and $R(u)$ is the Rayleigh quotient of $u$ defined by
\begin{equation*}
R(u):=\frac{\|\grad u\|_{L^2(M)}^2}{\|u\|_{L^2(M)}^2}.
\end{equation*}
(The $\grad u$ is defined on almost everywhere in $M$ because $u\in \Lip(M)$ is differentiable at almost everywhere in $M$ by the Rademacher's theorem.)
\end{dfn}
\begin{prop}[{\cite[Proposition 2.38]{ShioyaMMG}}]
\label{eigenSep}
Let $M$ be a compact Riemannian manifold, and let $k\in \N$. For any $\ka_0, \ka_1,\ldots,\ka_k>0$, we have 
\begin{equation*}
\Sep(M;\ka_0, \ka_1,\ldots,\ka_k)\le \frac{2}{\sqrt{\la_k(M)\min_{i=0,1,\ldots,k}\ka_i}}.
\end{equation*}
\end{prop}
\begin{dfn}[$k$-L\'{e}vy family]
Let $k\in \N$. A sequence of mm-spaces $\{X_n\}_{n\in \N}$ is a \emph{$k$-L\'{e}vy family} if 
\begin{equation*}
\lim_{n\to \infty}\Sep(X_n;\ka_0, \ka_1,\ldots,\ka_k)=0
\end{equation*}
holds for any $\ka_0, \ka_1,\ldots,\ka_k>0$.
\end{dfn}
\begin{rem}
\label{eigenLevy}
By Proposition \ref{ObsSep}, a $1$-L\'{e}vy family is same as a L\'{e}vy family. A sequence of compact Riemannian manifolds $\{M_n\}_{n\in \N}$ satisfying $\lim_{n\to \infty}\la_k(M_n)=\infty$ for some $k\in \N$ is a $k$-L\'{e}vy family by Proposition \ref{eigenSep}.
\end{rem}
For an mm-space $X=(X,d_X,\mu_X)$ and $t>0$, we define $tX$ by 
\begin{equation*}
tX:=(X,td_X,\mu_X).
\end{equation*}
\begin{prop}[{\cite[Theorem 4.4]{Funano-Shioya}}, {\cite[Theorem 9.40]{ShioyaMMG}}]
\label{kLevy}
Let $k\in \N$ and let $\{X_n\}_{n\in \N}$ be a $k$-L\'{e}vy family. We have only one of the following. 
\begin{itemize}
\item $\{X_n\}_{n\in \N}$ is a L\'{e}vy family.
\item There exist a subsequence $\{n_i\}_{i\in \N}$, a sequence $\{t_i\}_{i\in \N}\subset (0,1]$ and an mm-space $Y$ with $|Y|\in \{2,\ldots,k\}$ such that $t_iX_{n_i}\concto Y$ holds.
\end{itemize}
\end{prop}
Next, we define some functions necessary to define the curvature dimension condition. 
\begin{dfn}
Let $\ka\in \R$. We define
\begin{equation*}
\om_{\ka}:=
\begin{cases}
\dfrac{\pi}{\sqrt{\ka}}\qquad&(\ka>0)\\
\infty\qquad &(\ka\le 0)
\end{cases}
\end{equation*}
and $s_{\ka}:[0,\infty)\to \R$ by
\begin{equation*}
s_{\ka}(\theta):=
\begin{cases}
\dfrac{\sin \sqrt{\ka}\theta}{\sqrt{\ka}\theta}\qquad&(\ka>0)\\
1\qquad &(\ka= 0)\\
\dfrac{\sinh \sqrt{-\ka}\theta}{\sqrt{-\ka}\theta}\qquad&(\ka<0).
\end{cases}
\end{equation*}
For $t\in [0,1]$, we define $\si_{\ka}^{(t)}:[0,\om_{\ka})\to [0,\infty]$ as
\begin{equation*}
\si_{\ka}^{(t)}(\theta):=
\begin{cases}
t\dfrac{s_{\ka}(t\theta)}{s_{\ka}(\theta)}\qquad&(\theta\in [0,\om_{\ka}))\\
\infty\qquad&(\theta\in [\om_{\ka},\infty))
\end{cases}
\end{equation*}
and $\tau_{K, N}^{(t)}:[0,\infty)\to [0,\infty]$ by
\begin{equation*}
\tau_{K, N}^{(t)}(\theta):=t^{\frac{1}{N}}\left(\si_{K/(N-1)}^{(t)}(\theta)\right)^{1-\frac{1}{N}}=
\begin{cases}
t\left(\dfrac{s_{K/(N-1)}(t\theta)}{s_{K/(N-1)}(\theta)}\right)^{1-\frac{1}{N}}&(\theta\in [0,\om_{K/(N-1)}))\\
\infty\ &(\theta\in [\om_{K/(N-1)},\infty))
\end{cases}
\end{equation*}
for $K\in \R$ and $N<0$. For simplicity of notation in the following, we put $\si_{K/N}^{(t),0}:=\si_{K/N}^{(1-t)}$, $\si_{K/N}^{(t),1}:=\si_{K/N}^{(t)}$, $\tau_{K, N}^{(t),0}:=\tau_{K, N}^{(1-t)}$, $\tau_{K, N}^{(t),1}:=\tau_{K, N}^{(t)}$.
\end{dfn}
We put $C_{K, N}^{(t), i}:=\sup_{\theta}\tau_{K, N}^{(t),i}\left(\theta\right)>0$ for $K\in \R$. This value varies depending on the range over which $\theta$ moves. For instance, when $K$ is nonnegative, then we have $C_{K, N}^{(t), i}\le t<\infty$, and when $K$ is negative and $\sup \theta<\pi\sqrt{(N-1)/K}$, then we have $C_{K, N}^{(t), i}<\infty$. In particular, we have $C_{K, N}^{(t), i}<\infty$ for any $N<0$ if $\sup \theta<\pi/\sqrt{-K}$.
\begin{dfn}[R\'{e}nyi entropy]
Let $X$ be a topological space and let $\mu$ be an element in $\P(X)$. For $N<0$, we define $S_{N,\,\mu}:\P(X)\to [1,\infty]$ by
\begin{equation*}
S_{N,\,\mu}(\nu):=
\begin{cases}
\int_X\rho^{1-\frac{1}{N}}d\mu\qquad&(\nu=\rho \mu)\\
\infty\qquad&(\nu\nll \mu).
\end{cases}
\end{equation*}
We call it the \emph{R\'{e}nyi entropy} of $\nu$ with respect to $\mu$. Also, we define 
\begin{equation*}
\D(S_{N,\,\mu}):=\{\nu\in \P(X)\mid S_{N,\,\mu}(\nu)<\infty\}.
\end{equation*}
\end{dfn}
\begin{dfn}[cf. {\cite[Definition 4.4]{Ohta}}) (Curvature-dimension condition]
Let $(X,d_X, \mu_X)$ be an mm-space and let $K\in \R$, $N<0$. We say that $X$ is a \emph{$\CD(K,N)$ space} (or satisfies \emph{$\CD(K,N)$ condition}) if for any $\nu_0=\rho_0\mu_X, \nu_1=\rho_1\mu_X\in \P_2(X)\cap \D(S_{N,\,\mu_X})$, there exist a geodesic $\{\nu_t\}_{t\in [0,1]}$ in $(\P_2(X),W_2)$ from $\nu_0$ to $\nu_1$ and $\pi\in \Opt(\nu_0,\nu_1)$ such that we have  
\begin{equation}
\label{CDine}
S_{N',\mu_X}(\nu_t)\le \sum_{i=0}^1\int_{X\times X}\tau_{K,N'}^{(t),i}\left(d_X(x_0,x_1)\right)\rho_i(x_i)^{-\NN}d\pi(x_0, x_1)
\end{equation}
for any $t\in [0,1]$ and for any $N'\in[N,0)$ with $\nu_0, \nu_1\in \D(S_{N',\,\mu_X})$.
\end{dfn}
\begin{dfn}[cf. {\cite[Definition 4.5]{Ohta}}) (Reduced curvature-dimension condition]
Let $(X,d_X, \mu_X)$ be an mm-space and let $K\in \R$, $N<0$. We say that $X$ is a \emph{$\CD^*(K,N)$ space} (or satisfies \emph{$\CD^*(K,N)$ condition}) if for any $\nu_0=\rho_0\mu_X, \nu_1=\rho_1\mu_X\in \P_2(X)\cap \D(S_{N,\,\mu_X})$, there exist a geodesic $\{\nu_t\}_{t\in [0,1]}$ in $(\P_2(X),W_2)$ from $\nu_0$ to $\nu_1$ and $\pi\in \Opt(\nu_0,\nu_1)$ such that we have  
\begin{equation}
\label{CD*ine}
S_{N',\mu_X}(\nu_t)\le \sum_{i=0}^1\int_{X\times X}\si_{K/N'}^{(t),i}\left(d_X(x_0,x_1)\right)\rho_i(x_i)^{-\NN}d\pi(x_0, x_1)
\end{equation}
for any $t\in [0,1]$ and for any $N'\in[N,0)$ with $\nu_0, \nu_1\in \D(S_{N',\,\mu_X})$.
\end{dfn}
It is clear that $\CD(0,N)$ condition is equivalent to $\CD^*(0,N)$ condition for any $N<0$ and Proposition 4.7 in \cite{Ohta} implies that a $\CD(K,N)$ space is a $\CD^*(K,N)$ space for any $K\in \R$ and $N<0$.
\begin{ex}[{\cite[Theorem 1.1]{Milman}}]
\label{Milmansphere}
Let $n\ge 2$, $\al>0$ and let $g$ be the canonical Riemannian metric on $\Ss^n$ induced by $\R^{n+1}$. If we define a function $\ph:\Ss^n\to [0,\infty)$ for $x\in \R^{n+1}$ with $|x|<1$ by
\begin{equation*}
\ph(y):=\frac{c_x^{n,\al}}{|y-x|^{n+\al}}
\end{equation*}
where $c_x^{n,\al}>0$ is a normalising constant, then, the mm-space $(\Ss^n,d_g,\ph\vol_g)$ where $d_g$ is the Riemannian distance induced by $g$ is a $\CD(n-1-\frac{n+\al}{4}, -\al)$ space.
\end{ex}
\begin{prop}
\label{CDscale}
Let $t>0$. If an mm-space $X$ is a $\CD(K,N)$ space, $tX$ is a $\CD(t^{-2}K,N)$ space.
\end{prop}
\begin{prop}[{\cite[Theorem 4.10]{Ohta}}]
Let $(M^n, g, e^{-f}\vol_g)$ be a weighted Riemannian manifold. For any $K\in \R$ and $N<0$, the following are equivalent.
\begin{itemize}
\item $\Ric_N\ge Kg$.
\item $(M^n, d_g, e^{-f}\vol_g)$ is a $\CD(K,N)$ space.
\item $(M^n, d_g, e^{-f}\vol_g)$ is a $\CD^*(K,N)$ space.
\end{itemize}
\end{prop}
\begin{dfn}[{\cite[Section 2.1]{Ohta}}]
Let $(M^n,g)$ be an $n$-dimensional Riemannian manifold and let $f\in C^{\infty}(M^n)$. We say that $f$ is \emph{$(K,N)$-convex} for some $K\in \R$ and $N<0$ if the condition
\begin{equation*}
\Hess f_N\ge -\frac{K}{N}f_N\cdot g
\end{equation*}
holds where $f_N:=\exp(-f/N)\in C^{\infty}(M^n)$.
\end{dfn}
\begin{prop}[{\cite[Corollary 4.12]{Ohta}}]
\label{mfdCD}
Let $n\in \N$, $K_1,K_2\in \R$ and $N_1,N_2\in \R$ with $N_1\ge n$ and $N_2<-N_1$. If $(M^n,g,\mu_M)$ is an $n$-dimensional complete weighted Riemannian manifold with $\Ric_{N_1}\ge K_1g$ and $f$ is a smooth $(K_2,N_2)$-convex function on $M^n$, then $(M^n,g,e^{-f}\mu_M)$ is a $\CD(K_1+K_2,N_1+N_2)$ space.
\end{prop}
\section{Some results on the estimates of the observable diameter}
The next theorems and corollary are extensions of Proposition 9.26 and Corollary 9.27 in \cite{ShioyaMMG} to $\CD(K,N)$ (and $\CD^*(K,N)$) spaces for $N<0$.
\begin{thm}
\label{thminfty}
Let $X$ be a $\CD(K,N)$ space for some $K>0$ and $N<0$. For any $\ka_0, \ka_1>0$ with $\ka_0+\ka_1<1$ and for any $\ka \in (0,1)$, 
\begin{enumerate}
\item $\displaystyle \Sep(X;\ka_0,\ka_1)\le 2\sqrt{\frac{1-N}{K}}\cosh^{-1}\left(\left(\frac{\ka_0^{1/N}+\ka_1^{1/N}}{2}\right)^{\frac{-N}{1-N}}\right)$
\item $\displaystyle \Obsdiam(X;-\ka)\le 2\sqrt{\frac{1-N}{K}}\cosh^{-1}\left(\left(2\ka^{-1}\right)^{\frac{1}{1-N}}\right)$
\end{enumerate}
hold, where $\cosh^{-1}$ is the inverse function of $\cosh$ given by $\cosh^{-1}(x)=\log(x+\sqrt{x^2-1})$ as $x\ge 1$.
\end{thm}
We obtain similar results for $\CD^*(K,N)$ spaces.
\begin{thm}
\label{thminfty*}
Let $X$ be a $\CD^*(K,N)$ space for some $K>0$ and $N<0$. For any $\ka_0, \ka_1>0$ with $\ka_0+\ka_1<1$ and for any $\ka \in (0,1)$, 
\begin{enumerate}
\item[(3)] $\displaystyle \Sep(X;\ka_0,\ka_1)\le 2\sqrt{\frac{-N}{K}}\cosh^{-1}\left(\frac{\ka_0^{1/N}+\ka_1^{1/N}}{2}\right)$
\item[(4)] $\displaystyle \Obsdiam(X;-\ka)\le 2\sqrt{\frac{-N}{K}}\cosh^{-1}\left((2\ka^{-1})^{-\frac{1}{N}}\right)$
\end{enumerate}
hold.
\end{thm}
Since Theorem \ref{thminfty*} can be proved in the same way as Theorem \ref{thminfty} by replacing $\tau_{K,N}^{(t)}$ by $\si_{K/N}^{(t)}$, we will not write the proof.
\begin{proof}[Proof of Theorem $\ref{thminfty}$]
(2) follows from (1) and Proposition \ref{ObsSep}. We prove only (1). We take any $A_0,A_1\in \B_X$ with $\mu_X(A_0)\ge \ka_0$ and $\mu_X(A_1)\ge \ka_1$. For $R>0$, $i=0,1$ and a fixed $a\in X$, we put
\begin{align*}
A_i(R):=&A_i\cap B_X(a,R)\\
\ka_i(R):=&\mu_X(A_i(R)).
\end{align*}
Since $\ka_i(R)$ converges to $\mu_X(A_i)>0$ as $R\to\infty$, we have $\ka_i(R)>0$ for sufficiently large $R>0$. For such $R>0$ and $i=0,1$, we put
\begin{equation*}
\nu_{i,R}:=\frac{\1_{A_i(R)}}{\ka_i(R)}\mu_X.
\end{equation*}
Then, the finiteness of $S_{N,\,\mu_X}(\nu_{i,R})$ and the boundedness of $A_i(R)$ give $\nu_{i,R}\in \P_2(X)\cap \D(S_{N,\,\mu_X})$. Because $X$ is a $\CD(K,N)$ space, there exist a geodesic $\{\nu_{t,R}\}_{t\in [0,1]}$ on $(\P_2(X),W_2)$ from $\nu_{0,R}$ to $\nu_{1,R}$ and $\pi\in \Opt(\nu_{0,R},\nu_{1,R})$ such that for any $t\in [0,1]$,
\begin{equation*}
S_{N,\,\mu_X}(\nu_{t,R})\le \sum_{i=0}^1\int_{X\times X}\tau_{K,N}^{(t),i}\left(d_X(x_0,x_1)\right)\left(\frac{\1_{A_i(R)}(x_i)}{\ka_i(R)}\right)^{-\frac{1}{N}}d\pi(x_0, x_1)
\end{equation*}
holds. If $t=1/2$, we have
\begin{align*}
\tau_{K,N}^{(t),0}(\theta)=\tau_{K,N}^{(t),1}(\theta)&=\frac{1}{2}\left(\dfrac{2\sinh \left(\frac{\theta}{2}\sqrt{\frac{K}{1-N}}\right)}{\sinh \sqrt{\frac{K}{1-N}}\theta}\right)^{1-\frac{1}{N}}
=\frac{1}{2}\left(\cosh \left(\frac{\theta}{2}\sqrt{\frac{K}{1-N}}\right)\right)^{\frac{1}{N}-1}
\end{align*}
holds, and for any $x_0\in A_0(R), x_1\in A_1(R)$, $A_i(R)\subset A_i$ implies
\begin{equation*}
d_X(x_0,x_1)\ge \dist(A_0(R),A_1(R))\ge \dist(A_0,A_1).
\end{equation*}
Hence, $\pi(X\times X\setminus (A_0(R)\times A_1(R)))=0$ implies
\begin{align*}
1&\le S_{N,\,\mu_X}(\nu_{1/2,R})\\
&\le \sum_{i=0}^1\int_{X\times X}\tau_{K,N}^{(1/2),i}\left(d_X(x_0,x_1)\right)\left(\frac{\1_{A_i(R)}(x_i)}{\ka_i(R)}\right)^{-\frac{1}{N}}d\pi(x_0, x_1)\\
&\le \sum_{i=0}^1\int_{A_0(R)\times A_1(R)}\frac{1}{2}\left(\cosh \left(\frac{\dist(A_0,A_1)}{2}\sqrt{\frac{K}{1-N}}\right)\right)^{\frac{1}{N}-1}\ka_i(R)^{1/N}d\pi(x_0, x_1)\\
&= \frac{\ka_0(R)^{1/N}+\ka_1(R)^{1/N}}{2}\left(\cosh \left(\frac{\dist(A_0,A_1)}{2}\sqrt{\frac{K}{1-N}}\right)\right)^{\frac{1}{N}-1}.
\end{align*}
We take $R\to \infty$ to obtain
\begin{equation*}
1\le \frac{\ka_0^{1/N}+\ka_1^{1/N}}{2}\left(\cosh \left(\frac{\dist(A_0,A_1)}{2}\sqrt{\frac{K}{1-N}}\right)\right)^{\frac{1}{N}-1}.
\end{equation*}
Hence, we have
\begin{equation*}
\dist(A_0,A_1)\le 2\sqrt{\frac{1-N}{K}}\cosh^{-1}\left(\left(\frac{\ka_0^{1/N}+\ka_1^{1/N}}{2}\right)^{\frac{-N}{1-N}}\right).
\end{equation*}
Thus, we have (1) by the arbitrariness of $A_0$ and $A_1$.
\end{proof}
\begin{cor}
\label{corinfty}
Let $\{K_n\}_{n\in \N}\subset \R$ and $\{N_n\}_{n\in\N}\subset (-\infty,0)$. Suppose that $K_n$ diverges to infinity. If a sequence of mm-spaces $\{X_n\}_{n\in \N}$ satisfies either of the following conditions, then, $\{X_n\}_{n\in \N}$ is a L\'{e}vy family.
\begin{itemize}
\item $X_n$ is a $\CD(K_n,N_n)$ space.
\item $X_n$ is a $\CD^*(K_n,N_n)$ space and $K_n\cdot N_n\to -\infty$ as $n\to \infty$.
\end{itemize}
\end{cor}
\begin{proof}
We take any $\ka\in (0,1)$. For any sufficiently large $n$, we have $K_n>0$. If $X_n$ is a $\CD(K_n,N_n)$ space, then
\begin{align*}
\Obsdiam(X_n;-\ka)&\le2\sqrt{\frac{1-N_n}{K_n}}\cosh^{-1}\left(\left(2\ka^{-1}\right)^{\frac{1}{1-N_n}}\right)\\
&\le 2\sqrt{\frac{1-N_n}{K_n}}\cdot\sqrt{2\left(\left(2\ka^{-1}\right)^{\frac{1}{1-N_n}}-1\right)}
\end{align*}
holds by Theorem \ref{thminfty} (2) and $\cosh^{-1}(x)\le \sqrt{2(x-1)}$ as $x\ge1$. If we put $a_n=(1-N_n)^{-1}\in (0,1)$, then, we have $((2\ka^{-1})^{a_n}-1)/a_n<2\ka^{-1}-1$ since the function $x\mapsto ((2\ka^{-1})^x-1)/x$ is increasing by $2\ka^{-1}>2$. Thus, we get 
\begin{align*}
\Obsdiam(X_n;-\ka)&\le \frac{2\sqrt{2}}{\sqrt{K_n}}\cdot\sqrt{\frac{\left(\left(2\ka^{-1}\right)^{a_n}-1\right)}{a_n}}
\le\frac{2\sqrt{2}}{\sqrt{K_n}}\cdot\sqrt{2\ka^{-1}-1}.
\end{align*}
By $K_n\to \infty$, we have $\Obsdiam(X_n;-\ka)\to 0$.

On the other hand, if $X_n$ is a $\CD^*(K_n,N_n)$ space and $K_n\cdot N_n\to -\infty$ as $n\to \infty$, then, we can assume $N_n\ge -1$ for any $n\in \N$ since the $\CD^*(K_n,N_n)$ condition implies the $\CD^*(K_n,(-1)\lor N_n)$ condition. By Theorem \ref{thminfty*} (4) and $\cosh^{-1}(a^x)\le x\log a+\log2$ as $a>1$ and $x>0$. Hence, we obtain
\begin{align*}
\Obsdiam(X_n;-\ka)&\le2\sqrt{\frac{-N_n}{K_n}}\left(-\frac{1}{N_n}\log(2\ka^{-1})+\log2\right)\\
&\le \frac{2\log(2\ka^{-1})}{\sqrt{-K_n\cdot N_n}}+\frac{2\log2}{\sqrt{K_n}}.
\end{align*}
By $K_n\to \infty$ and $K_n\cdot N_n\to -\infty$, we have $\Obsdiam(X_n;-\ka)\to 0$. Thus, $\{X_n\}_{n\in \N}$ is a L\'{e}vy family in both cases.
\end{proof}
\begin{rem}
The mm-space $(\Ss^n,d_g,\ph\vol_g)$ given in Example \ref{Milmansphere} is a $\CD(n-1-\frac{n+\al}{4}, -\al)$ space, and by
\begin{equation*}
\lim_{n\to \infty}K_n=\lim_{n\to \infty}\left(n-1-\frac{n+\al}{4}\right)=\infty,
\end{equation*}
Corollary \ref{corinfty} implies that $\{(\Ss^n,d_g,\ph\vol_g)\}_{n\in \N}$ is a L\'{e}vy family.
\end{rem}

\section{Estimates of entropy}
In this section, we give several lemmas which are necessary for the proof of the main theorem. Since these lemmas are rewritings of the propositions in the references for the case that the entropy is replaced by the R\'{e}nyi entropy for $N<0$, most of the proofs can be done in the same way. Thus, we shall give no proof or only brief outlines of the proofs.
\begin{lem}[cf. {\cite[Proposition 4.8]{Magnabosco-Rigoni-Sosa}}]
\label{OY18'}
Let $N'$ be a negative number. The R\'{e}nyi entropy $S_{N',\cdot}(\cdot):\P(X)\times\P(X)\to [1,\infty]$ is lower-semicontinuous with respect to the weak convergence topology.
\end{lem}
\begin{lem}[cf. {\cite[Proposition 4.1]{Gigli-Mondino-Savare}}]
\label{OY19'}
For any $N'<0, \ep>0$ and $E\in (0,\infty)$, there exists $\de>0$ such that $\nu(X\setminus \tX)<\ep$ if $\mu,\nu \in \P(X), \tX\in \B_X$ satisfy $\mu(X\setminus  \tX)<\de$ and $S_{N',\mu}(\nu)\le E$. In particular, for any $\{\mu_n\}_{n=1}^{\infty}, \{\nu_n\}_{n=1}^{\infty}\subset \P(X)$, if $\{\mu_n\}_{n=1}^{\infty}$ is tight and $\sup_{n\in \N}S_{N',\mu_n}(\nu_n)$ $<\infty$, then $\{\nu_n\}_{n=1}^{\infty}$ is also tight.
\end{lem}
\begin{lem}[cf. {\cite[Lemma 9.15]{ShioyaMMG}}]
\label{OY27'}
Let $p:X\to Y$ be a Borel map between two complete separable metric spaces $X$ and $Y$. If $\mu, \nu \in \P(X)$ and $\nu \ll \mu$, then $p_*\nu \ll p_*\mu$ and
\begin{equation*}
S_{N',\, p_*\mu}(p_*\nu)\le S_{N',\, \mu}(\nu)
\end{equation*}
for any $N'<0$.
\end{lem}
\begin{lem}[cf. {\cite[Lemma 28]{Ozawa-Yokota}}]
\label{OY28'}
Let $N'$ be a negative number. For any $\mu, \nu\in \P(X)$ and $B\in \B_X$ with $\nu(B)>0$. Then,
\begin{equation*}
\nu(B)^{1-\NN}S_{N',\mu}(\nu_B)\le S_{N', \mu}(\nu).
\end{equation*}
\end{lem}
\begin{lem}[{\cite[Lemma 29]{Ozawa-Yokota}}]
\label{OY29}
Let $Z_n$, $Y$ be complete separable metric spaces and $q_n:Z_n\to Y$ be a Borel map. Suppose that $B_n\subset Z_n$, $\mu_n\in \P(Z_n)$ and $\nu\in \P(Y)$ satisfy $\mu_n(B_n)\to 1$  and $(q_n)_*\mu_n\wto \nu$ as $n\to \infty$. Then $(q_n)_*\left((\mu_n)_{B_n}\right)$ converges to $\nu$ weakly. If $\nu\in \P_2(Y)$ and $W_2((q_n)_*\mu_n, \nu)\to 0$, then $W_2((q_n)_*\left((\mu_n)_{B_n}\right), \nu)\to 0$ as $n\to \infty$.
\end{lem}
\begin{lem}[cf. {\cite[Corollary 30]{Ozawa-Yokota}}]
\label{OY30'}
Let $N'$ be a negative number and let $\mu, \nu$ be elements in $\P(X)$.  If $\{B_n\}_{n=1}^{\infty}\subset \B_X$ satisfies $\nu(B_n)\to 1$, then, $\nu_{B_n}$ converges to $\nu$ weakly and $S_{N',\mu}(\nu_{B_n})$ converges to $S_{N',\mu}(\nu)$ as $n\to \infty$.
\end{lem}
\begin{lem}[cf. {\cite[Lemma 31]{Ozawa-Yokota}}]
\label{OY31'}
Let $\mu\in \P(X)$, $\nu\in \P_2(X)$ and $N'<0$. If a countable family of mutually disjoint subsets $\{B_j\}_{j\in \J}\subset \B_X$ satisfies $\nu(X\setminus \bigcup_{j\in \J}B_j)=0$ and $\mu(B_j)>0$ for any $j\in \J$, then we have
\begin{equation*}
W_2(\nu, \unu)\le 2D\quad \text{and}\quad S_{N', \mu}(\nu)\ge S_{N', \mu}(\unu),
\end{equation*}
where $D:=\sup_{j\in \J}\diam B_j$ and $\unu:=\sum_{j\in \J}\nu(B_j)\mu_{B_j}\in \P(X)$. 
\end{lem}
\begin{lem}[{\cite[Lemma 32]{Ozawa-Yokota}}]
\label{OY32}
Let $(X,d_X)$ be a metric space. If $\{\nu_n\}_{n=1}^{\infty}, \{\nu_n'\}_{n=1}^{\infty}\subset \P(X)$ and $\nu \in \P(X)$ satisfy $W_2(\nu_n,\nu_n')\to 0$ and $\nu_n\wto \nu$ as $n\to \infty$, then, we have $\nu_n'\wto \nu$.
\end{lem}
\begin{lem}[cf. {\cite[Lemma 38]{Ozawa-Yokota}}]
\label{OY38'}
Let $X$ be a topological space and $N'<0$. Suppose that $\mu$, $\nu_n=\rho_n\mu$, $\nu=\rho\mu\in \P(X)$ satisfy $\nu_n\wto \nu$ and $\limsup_{n\to \infty}S_{N', \mu}(\nu_n)\le S_{N', \mu}(\nu)<\infty$, then, we have $\int_{X}|\rho_n-\rho|d\mu\to 0$. 
\end{lem}
\begin{lem}[{\cite[Lemma 42]{Ozawa-Yokota}}]
\label{OY42}
Let $(Y,d_Y,\mu_Y)$ be an mm-space. For any $\de>0$ and $S\subset Y$ with $\mu_Y(S)>0$, there exists a countable family of mutually disjoint subsets $\{B_j\}_{j\in \J}\subset \B_Y$ such that $\diam B_j\le \de$, $\mu_Y(S\cap B_j)>0$, $\mu_Y(\del B_j)=0$ for any $j\in \J$ and $\mu_Y(S\setminus \bigcup_{j\in \J}B_j)=0$.
\end{lem}
\begin{dfn}
Let $(X,d_X,\mu_X)$ be an mm-space. We define 
\begin{equation*}
\P^{ac}(X):=\{\nu\in \P(X)\mid \nu\ll\mu_X\},\quad \P_2^{ac}(X):=\P^{ac}(X)\cap \P_2(X)
\end{equation*}
\begin{equation*}
\P_{cb}(X):=\left\{\nu\in \P^{ac}(X)\left| \text{ $\supp \nu$ is compact, $\frac{d\nu}{d\mu_X}$ is bounded}\right.\right\}
\end{equation*}
\end{dfn}
\begin{rem}
For an mm-space $(X,d_X,\mu_X)$ and $B\in \B_X$ with $\mu_X(B)>0$, $\left(\mu_X\right)_B$ is simply denoted as $\mu_B$.
\end{rem}
\begin{lem}[cf. {\cite[Lemma 44]{Ozawa-Yokota}}]
\label{OY44'}
Let $\nu_0$ and $\nu_1$ be elements in $\P^{ac}(Y)$. For any $m\in \N$, let $\{B_{j,m}\}_{j\in \J'_m}$ be a countable family given by Lemma $\ref{OY42}$ as $\de=m^{-1}$ and $S=\supp \nu_0 \cup \supp \nu_1$, and let $\J_m\subset \J'_m$ be a finite set such that $\mu_Y(S\setminus \bigcup_{j\in \J_m}B_{j,m})<m^{-1}$. If we put
\begin{align*}
U_m:=\bigcup_{j\in \J_m}B_{j,m}\subset Y, \quad \unu_i^m:=\sum_{j\in \J_m}\frac{\nu_i(B_{j, m})}{\nu_i(U_m)}\mu_{B_{j,m}}\in \P_2(Y),
\end{align*}
then,  $\nu_i(U_m)$ converges to $1$, $\unu_i^m$ converges to $\nu_i$ weakly and $S_{N',\mu_Y}(\unu_i^m)$ converges to $S_{N',\mu_Y}(\nu_i)$ as $m\to \infty$. In addition, if $\nu_i$ is an element in $\P_2(Y)$, then, $\unu_i^m$ $W_2$-converges to $\nu_i$.
\end{lem}
Hereafter, unless otherwise noted, let $X_n$ and $Y$ be mm-spaces and let $p_n$, $\ep_n$, $\tX_n$ be given by Proposition \ref{OY14} under the assumption $X_n\concto Y$.
\begin{lem}[cf. {\cite[Lemma 35]{Ozawa-Yokota}}]
\label{OY35'}
Suppose that $X_n$ concentrates to $Y$. If $\nu_n, \nu_n'\in \P_2(X_n)$ and $\nu, \nu'\in \P(Y)$ satisfy $(p_n)_*\nu_n\wto \nu$, $(p_n)_*\nu_n'\wto \nu'$ and $\limsup_{n\to \infty}\left(S_{N',\mu_{X_n}}(\nu_n)+S_{N',\mu_{X_n}}(\nu_n')\right)<\infty$ for some $N'<0$, then, we have 
\begin{equation*}
W_2(\nu,\nu')\le \liminf_{n\to \infty}W_2(\nu_n, \nu_n').
\end{equation*}
\end{lem}
\begin{lem}[{\cite[Lemma 3.13]{Funano-Shioya}}, {\cite[Lemma 9.33]{ShioyaMMG}}]
\label{OY40}
Suppose that $X_n$ concentrates to $Y$. For any sufficiently small $\de>0$ and any $B_0, B_1\in \B_Y$ with
\begin{equation*}
\diam B_i\le \de,\quad \mu_Y(B_i)>0, \quad \text{and} \quad \mu_Y(\del B_i)=0
\end{equation*}
for $i=0,1$, there exist $\xi_0^n, \xi_1^n\in \P(X_n)$ and $\pi_{01}^n\in \Cpl(\xi_0^n,\xi_1^n)$ such that for all $n\in \N$ large enough, we have
\begin{enumerate}
\item $\xi_i^n\le (1+\theta_1(\de^{1/2}))\mu_{\tB_i}$\quad $(i=0,1)$
\item $d_{X_n}(\tB_0, \tB_1)\ge d_Y(B_0,B_1)-\ep_n$
\item $\supp \pi_{01}^n\subset \{(x,x')\in X_n\times X_n\mid d_{X_n}(x,x')\le d_Y(B_0,B_1)+\de^{1/2}\}$
\item $-\ep_n\le W_2(\xi_0^n, \xi_1^n)-d_Y(B_0,B_1)\le \de^{1/2}$
\end{enumerate}
where $\tB_i:=p_n^{-1}(B_i)\cap \tX_n$ and $\theta_1$ is a function $\theta_1:[0,\infty)\to [0,\infty)$ with $\theta_1(x)\to 0$ as $x\to 0$.
\end{lem}
\begin{lem}[cf. {\cite[Lemma 43]{Ozawa-Yokota}}]
\label{OY43'}
Suppose that $X_n$ concentrates to $Y$, and let $\nu_0, \nu_1 \in \P_2^{ac}(Y)$. If the same condition in Lemma $\ref{OY44'}$, i.e. for any $m\in \N$, let $\{B_{j,m}\}_{j\in \J'_m}$ be a countable family given by Lemma $\ref{OY42}$ as $\de=m^{-1}$ and $S=\supp \nu_0 \cup \supp \nu_1$, and let $\J_m\subset \J'_m$ be a finite set such that $\mu_Y(S\setminus \bigcup_{j\in \J_m}B_{j,m})<m^{-1}$, and we put
\begin{align*}
U_m:=\bigcup_{j\in \J_m}B_{j,m}\subset Y, \quad \unu_i^m:=\sum_{j\in \J_m}\frac{\nu_i(B_{j, m})}{\nu_i(U_m)}\mu_{B_{j,m}}\in \P_2(Y).
\end{align*}
Then, for sufficiently large $n\in \N$, there exist $\nu_0^{mn}, \nu_1^{mn} \in \P(X_n)$ such that
\begin{equation*}
\limsup_{n\to \infty}W_2((p_n)_*\nu_i^{mn}, \unu_i^m)\le\theta_2(m^{-1})
\end{equation*}
\begin{equation*}
\limsup_{n\to \infty}\left|W_2(\nu_0^{mn},\nu_1^{mn})-W_2(\unu_0^m,\unu_1^m)\right|\le \theta_2(m^{-1})
\end{equation*}
hold for $i=0,1$. In addition, for any $N'<0$ with $\nu_0, \nu_1\in \D(S_{N',\mu_Y})$, we have
\begin{equation*}
\limsup_{n\to \infty}\left|S_{N',\mu_{X_n}}(\nu_i^{mn})-S_{N',\mu_Y}(\unu_i^m)\right|\le \theta_2(m^{-1})
\end{equation*}
where $\theta_2$ is a function $\theta_2:[0,\infty)\to [0,\infty)$ with $\theta_2(x)\to 0$ as $x\to 0$.
\end{lem}
\begin{proof}[Sketch of the proof] We take any $(j,k)\in \J_m\times \J_m$. By applying Lemma \ref{OY40} for $B_0=B_{j,m}$, $B_1=B_{k,m}$ and $\de=m^{-1}$, we get $\jkmn \in \P(X_n)$ satisfying
\begin{equation*}
\jkmn \le (1+\theta_1(m^{-1}))\mu_{\tB_{j,m}} \text{\quad and \quad} W_2(\jkmn,\kjmn)\le d_Y(B_{j,m},B_{k,m})+m^{-1/2}.
\end{equation*}
Then, the tightness of $\{(p_n)_*\jkmn\}_{n\in \N}$ implies that there exists $\jkm\in \P(Y)$ such that $(p_n)_*\jkmn$ converges to $\jkm$ weakly as $n\to \infty$ for any $(j,k)\in \J_m\times \J_m$. We fix a $\pi_m\in \Opt(\unu_0^m,\unu_1^n)$ and put
\begin{equation*}
w_{jk}^0:=\pi_m(B_{j,m}\times B_{k,m}), \quad w_{jk}^1:=w_{kj}^0
\end{equation*}
\begin{equation*}
\nu_i^{mn}:=\sum_{j,k\in \J_m}w_{jk}^i\jkmn \in \P_2(X_n), \quad  \nu_i^m:=\sum_{j,k\in \J_m}w_{jk}^i\jkm \in \P_2(Y)
\end{equation*}
for $i=0,1$. Considering $(p_n)_*\nu_i^{mn}\wto \nu_i^m$, $S_{N',\mu_Y}(\unu_i^m)\le S_{N',\mu_Y}(\nu_i^m)$, etc., we can prove this lemma in the same way as in \cite{Ozawa-Yokota}.
\end{proof}

\section{Proof of Main Theorem \ref{mainthm1}}
In this section, 
we give the proof of Main Theorem \ref{mainthm1}. For this purpose, we first show the equivalent condition of the $\CD(K,N)$ condition.
\begin{lem}
\label{wCD}
Let $K\in \R$, $N<0$, and let $Y$ be an mm-space with $\diam Y<\pi/\sqrt{-K}$ $(K<0)$. Then, 
the following are equivalent.
\begin{enumerate}
\item[(A)] $Y$ is a $\CD(K, N)$ space.
\item[(B)] For any $\nu_0=\rho_0\mu_Y$, $\nu_1=\rho_1\mu_Y\in \P_{cb}(Y)$, there exist $\{\nu_t\}_{t\in [0,1]\cap\Q}\subset \P_2(Y)$ and $\pi\in \Opt(\nu_0, \nu_1)$ such that
\begin{equation*}
W_2(\nu_s, \nu_t)\le |t-s|W_2(\nu_0, \nu_1)
\end{equation*}
\begin{equation*}
S_{N',\mu_Y}(\nu_t)\le \sum_{i=0}^1\int_{Y\times Y}\tau_{K, N'}^{(t),i}\left(d_Y(y_0,y_1)\right)\rho_i(y_i)^{-\NN}d\pi(y_0, y_1)
\end{equation*} 
for any $s, t\in[0,1]\cap\Q$ and $N'\in [N,0)$.
\end{enumerate}
\end{lem}
\begin{proof} 
(A) $\Rightarrow$ (B) is obvious by $\P_{cb}(Y)\subset \P_2(Y)\cap \D(S_{N, \mu_Y})$. Hence, it suffices to prove (B) $\Rightarrow$ (A). We take any $\nu_0=\rho_0\mu_Y$, $\nu_1=\rho_1\mu_Y\in \P_2(Y)\cap \D(S_{N, \mu_Y})$. Then, for any $n\in \N$ and $i\in\{0,1\}$, we put
\begin{equation*}
\nu_i^n=\rho_i^n\mu_Y,\quad \rho_i^n:=(c_{i,n})^{-1}(\rho_i\land n)\1_{K_n},\quad c_{i,n}:=\int_{K_n}(\rho_i\land n)\,d\mu_Y
\end{equation*}
where $K_n$ is a compact set satisfying $K_n\subset K_{n+1}$ and $\mu_Y(Y\setminus K_n)\to 0$. Then, $\nu_i^n$ is an element in $\P_{cb}(Y)$ and $W_2$-converges to $\nu_i$ for any $i=0,1$. 
Now, by the assumption of (B), there exist $\{\nu_t^n\}_{t\in [0,1]\cap\Q}\subset \P_2(Y)$ and $\pi^n\in \Opt(\nu_0^n, \nu_1^n)$ such that 
\begin{equation*}
W_2(\nu_s^n, \nu_t^n)\le |t-s|W_2(\nu_0^n, \nu_1^n)
\end{equation*}
\begin{equation*}
S_{N',\mu_Y}(\nu_t^n)\le \sum_{i=0}^1\int_{Y\times Y}\tau_{K, N'}^{(t),i}\left(d_Y(y_0,y_1)\right)\rho_i^n(y_i)^{-\NN}d\pi^n(y_0, y_1)
\end{equation*}
for any $s, t\in[0,1]\cap\Q$ and for any $N'\in [N,0)$. First, from Proposition \ref{Cpltight}, we have $\{\pi^n\}_{n=1}^{\infty}$ is tight, there exist $\pi\in \P(Y\times Y)$ and a subsequence $\{n_k^0\}_{k\in \N}$ such that $\pi^{n_k^0}\wto \pi$, and  we also know that $\pi$ is an element in $\Opt(\nu_0, \nu_1)$ by Proposition \ref{Wplsc}. For any $N'\in [N, 0)$ with $\nu_0, \nu_1\in \D(S_{N', \mu_Y})$,
\begin{align*}
S_{N', \mu_Y}(\nu_t^{n_k^0})\le&\sum_{i=0}^1\int_{Y\times Y}\tau_{K, N'}^{(t),i}\left(d_Y(y_0,y_1)\right)\rho_i^{n_k^0}(y_i)^{-\NN}d\pi^{n_k^0}(y_0, y_1)\\
\le&\sum_{i=0}^1\left(c_{i,n_k^0}\right)^{\NN}\int_{Y\times Y}\tau_{K, N'}^{(t),i}\left(d_Y(y_0,y_1)\right)\rho_i(y_i)^{-\NN}d\pi^{n_k^0}(y_0, y_1)\\
\end{align*}
for any $t\in [0,1]\cap\Q$. Then, 
the condition (b) of Lemma \ref{wconvint} holds by $\nu_i^{n_k^0}\le (c_{i,n_k^0})^{-1}\nu_i$, $c_{i,n_k^0}\to 1$ as $k\to \infty$ and $C_{K, N'}^{(t),i}<\infty$. Thus, Lemma \ref{wconvint} implies
\begin{align*}
&\limsup_{k\to \infty}S_{N', \mu_Y}(\nu_t^{n_k^0})
\le\sum_{i=0}^1\int_{Y\times Y}\tau_{K, N'}^{(t),i}\left(d_Y(y_0,y_1)\right)\rho_i(y_i)^{-\NN}d\pi(y_0,y_1)
\end{align*}
holds. Now, since this right-hand side is finite, $\sup_{k\in \N}S_{N', \mu_Y}(\nu_t^{n_k^0})$ is finite. Thus, Lemma \ref{OY19'} gives that $\{\nu_t^{n_k^0}\}_{k\in \N}$ is tight. Now, let $(0,1)\cap\Q=\{t_l\mid l\in \N\}$, then, by the above argument, $\{\nu_{t_1}^{n_k^0}\}_{k\in \N}$ will be tight. Thus, by Prokhorov's theorem, there exist a subsequence $\{n_k^1\}_{k\in \N}\subset \{n_k^0\}_{k\in \N}$ and $\nu_{t_1}\in \P(Y)$ such that $\nu_{t_1}^{n_k^1}\wto \nu_{t_1}$. Similarily, because $\{\nu_{t_2}^{n_k^1}\}_{k\in \N}$ is tight, there exist a subsequence $\{n_k^2\}_{k\in \N}\subset \{n_k^1\}_{k\in \N}$ and $\nu_{t_2}\in \P(Y)$ such that $\nu_{t_2}^{n_k^2}\wto \nu_{t_2}$. By repeating this operation, we obtain $\{n_k^l\}_{k,l\in \N}$ and $\{\nu_{t_l}\}_{l\in \N}\cup\{\nu_0,\nu_1\}=\{\nu_t\}_{t\in [0,1]\cap\Q}$, while $n_k:=n_k^k$,  we obtain that for any $t\in [0,1]\cap\Q$, $\nu_t^{n_k}$ converges to $\nu_t$ weakly. Then, for any $s,t\in [0,1]\cap \Q$ and $N'\in [N, 0)$ with $\nu_0, \nu_1\in \D(S_{N', \mu_Y})$
\begin{align*}
W_2(\nu_s, \nu_t)&\le \liminf_{k\to \infty}W_2(\nu_s^{n_k}, \nu_t^{n_k})\qquad(\text{Lemma \ref{Wplsc}})\\
&\le \liminf_{k\to \infty}|t-s|W_2(\nu_0^{n_k}, \nu_1^{n_k})\\
&= |t-s|W_2(\nu_0, \nu_1)
\end{align*}
and
\begin{align*}
S_{N', \mu_Y}(\nu_t)\le& \liminf_{k\to \infty}S_{N', \mu_Y}(\nu_t^{n_k})\quad (\text{Lemma \ref{OY18'}})\\
\le& \limsup_{k\to \infty}S_{N', \mu_Y}(\nu_t^{n_k^0})\\
\le&\sum_{i=0}^1\int_{Y\times Y}\tau_{K, N'}^{(t),i}\left(d_Y(y_0,y_1)\right)\rho_i(y_i)^{-\NN}d\pi(y_0,y_1)
\end{align*}
holds. If we consider $\nu$ as a map $\nu:[0,1]\cap \Q\to \P_2(Y); t\mapsto \nu_t$, then $\nu$ is $W_2(\nu_0, \nu_1)$-Lipschitz. Since $(\P_2(Y), W_2)$ is a complete metric space, there is a unique extension of $\nu$ to $\overline{[0,1]\cap \Q}=[0,1]$ by Proposition \ref{closurextension}. This extension $\nu$ is a geodesic from $\nu_0$ to $\nu_1$ on $(\P_2(Y), W_2)$.
For any $t\in [0,1]$, by the definition of $\nu$, there exists $\{t_k\}_{k\in \N}\subset \Q$ such that $t_k\to t$ and $W_2(\nu_{t_k},\nu_t)\to 0$ as $k\to \infty$. Hence, the continuity of $\tau_{K, N'}^{(t)}$ for $t$ and Lebesgue's convergence theorem imply
\begin{align*}
S_{N', \mu_Y}(\nu_t)\le& \liminf_{k\to \infty}S_{N', \mu_Y}(\nu_{t_k})\\
\le& \limsup_{k\to \infty}\sum_{i=0}^1\int_{Y\times Y}\tau_{K, N'}^{(t_k),i}\left(d_Y(y_0,y_1)\right)\rho_i(y_i)^{-\NN}d\pi(y_0, y_1)\\
=& \sum_{i=0}^1\int_{Y\times Y}\tau_{K, N'}^{(t),i}\left(d_Y(y_0,y_1)\right)\rho_i(y_i)^{-\NN}d\pi(y_0, y_1).
\end{align*}
Thus $Y$ is a $\CD(K, N)$ space i.e. we have the condition (A).
\end{proof}
Lemma \ref{wCD} shows that it is sufficient to prove the condition (B) instead of proving $\CD(K,N)$ condition.
\begin{proof}[Proof of Main Theorem $\ref{mainthm1}$]
We can assume $\sup_{n\in \N}\diam X_n<\pi/\sqrt{-K}$ as $K<0$ by the assumption (\ref{diamassume}) and taking a subsequence first. Hence, this assumption gives $\diam Y<\pi/\sqrt{-K}$ and $C_{K,N'}^{(t)}<\infty$ for any $N'<0$ by Proposition \ref{diamlsc}. We take any $\nu_0, \nu_1\in \P_{cb}(Y)$. For this $\nu_i$ and sufficiently large $m,n\in \N$, let $\nu_i^{mn}, \unu_i^m$ be a family of measures constructed in Lemma \ref{OY43'}. Since $\nu_i^{mn}$ is an element in $\P_2(X_n)\cap \D(S_{N,\,\mu_Y})$ and $X_n$ satisfies the $\CD(K,N)$ condition, there are a geodesic $\{\nu_t^{mn}\}_{t\in[0,1]}\subset \P_2(X_n)$ from $\nu_0^{mn}$ to $\nu_1^{mn}$ on $(\P_2(X_n), W_2)$ and $\pi^{mn}\in \Opt(\nu_0^{mn}, \nu_1^{mn})$ such that 
\begin{equation*}
S_{N',\mu_{X_n}}(\nu_t^{mn})\le \sum_{i=0}^1\int_{X_n\times X_n}\tau_{K, N'}^{(t),i}\left(d_{X_n}(x_0,x_1)\right)\rho_i^{mn}(x_i)^{-\NN}d\pi^{mn}(x_0, x_1)
\end{equation*}
for any $t\in [0,1]$ and $N'\in [N,0)$. Then, Lemmas \ref{OY43'} and \ref{OY44'} imply
\begin{equation*}
\limsup_{n\to \infty}\left|W_2(\nu_0^{mn},\nu_1^{mn})-W_2(\unu_0^m,\unu_1^m)\right|\le \theta_2(m^{-1})
\end{equation*}
\begin{equation}
\label{K==0}
\limsup_{n\to \infty}\left|S_{N',\mu_{X_n}}(\nu_i^{mn})-S_{N',\mu_Y}(\unu_i^m)\right|\le \theta_2(m^{-1})\qquad(i=0,1)
\end{equation}
\begin{equation}
\label{K===0}
S_{N',\mu_Y}(\unu_i^m)\to S_{N',\mu_Y}(\nu_i)<\infty\qquad(m\to \infty)
\end{equation}
\begin{equation*}
W_2(\unu_i^m, \nu_i)\to 0\qquad(m\to \infty)
\end{equation*}
and we have
\begin{align*}
&\limsup_{m\to \infty}\limsup_{n\to \infty}W_2(\nu_0^{mn},\nu_1^{mn})\\
\le&\limsup_{m\to \infty}\left(W_2(\unu_0^m,\unu_1^m)+\theta_2(m^{-1})\right)\\
=&W_2(\nu_0,\nu_1).
\end{align*}
Then, since $(p_n\times p_n)_*\pi^{mn}$ is a coupling of $(p_n)_*\nu_0^{mn}$ and $(p_n)_*\nu_1^{mn}$ and $(p_n)_*\nu_i^{mn}$ converges to $\nu_i^m$ weakly, $\{(p_n\times p_n)_*\pi^{mn}\}_{n\in \N}$ is tight for each $m\in \N$. $(p_n\times p_n)_*\pi^{mn}$ converges to some $\pi^m\in \Cpl(\nu_0^m, \nu_1^m)$ weakly by taking a subsequence. Furthermore, since $\{\pi^m\}_{m\in \N}$ is tight by $\nu_i^m \wto \nu_i$, $\pi^m$ converges to some $\pi \in \Cpl(\nu_0, \nu_1)$ by taking a subsequence. Then,
\begin{align*}
\int_{Y\times Y}d_Y(y,y')^2d\pi(y,y')
\le&\liminf_{m\to \infty}\liminf_{n\to \infty}\int_{Y\times Y}d_Y(y,y')^2d((p_n\times p_n)_*\pi^{mn})(y,y')\\
=&
\liminf_{m\to \infty}\liminf_{n\to \infty}\int_{\tX_n\times \tX_n}d_Y(p_n(x),p_n(x'))^2d\pi^{mn}(x,x')\\
\le&\liminf_{m\to \infty}\liminf_{n\to \infty}\int_{\tX_n\times \tX_n}(d_{X_n}(x,x')+\ep_n)^2d\pi^{mn}(x,x')\\
\le&\liminf_{m\to \infty}\liminf_{n\to \infty}\left(W_2(\nu_0^{mn},\nu_1^{mn})+\ep_n\right)^2\\
\le&W_2(\nu_0,\nu_1)^2
\end{align*}
inplies that $\pi$ is an optimal coupling. Here, the following claim holds.
\begin{claim}
\label{claim}
Let $m,n$ be the above subsequences. Then
\begin{align*}
&\limsup_{m\to \infty}\limsup_{n\to \infty}\sum_{i=0}^1\int_{X_n\times X_n}\tau_{K, N'}^{(t),i}\left(d_{X_n}(x_0,x_1)\right)\rho_i^{mn}(x_i)^{-\NN}d\pi^{mn}(x_0, x_1)\\
\le&\sum_{i=0}^1\int_{Y\times Y}\tau_{K, N'}^{(t),i}\left(d_Y(y_0,y_1)\right)\rho_i(y_i)^{-\NN}d\pi(y_0, y_1)
\end{align*}
holds.
\end{claim}
The proof of this claim is given after the proof of the theorem. This claim and Lemma \ref{OY27'} imply
\begin{align}
&\limsup_{m\to \infty}\limsup_{n\to \infty}S_{N', (p_n)_*\mu_{X_n}}((p_n)_*\nu_t^{mn})\notag \\
\le&\limsup_{m\to \infty}\limsup_{n\to \infty}S_{N',\mu_{X_n}}(\nu_t^{mn})\notag \\
\le&\limsup_{m\to \infty}\limsup_{n\to \infty}\sum_{i=0}^1\int_{X_n\times X_n}\tau_{K, N'}^{(t),i}\left(d_{X_n}(x_0,x_1)\right)\rho_i^{mn}(x_i)^{-\NN}d\pi^{mn}(x_0, x_1)\notag \\
\label{claimK}
\begin{split}
\le&\sum_{i=0}^1\int_{Y\times Y}\tau_{K, N'}^{(t),i}\left(d_Y(y_0,y_1)\right)\rho_i(y_i)^{-\NN}d\pi(y_0, y_1)<\infty 
\end{split}
\end{align}
for any $t\in [0,1]$. Hence Lemma \ref{OY19'} implies that $\{(p_n)_*\nu_t^{mn}\}_{m,n\in \N}$ is tight for any $t\in [0,1]$ because $\{(p_n)_*\mu_{X_n}\}_{n\in \N}$ is tight. Here, we label rational numbers in $(0,1)$ as $\{t_l \mid l\in \N\}$ and fix any $m\in \N$. Because $\{(p_n)_*\nu_{t_1}^{mn}\}_{n\in \N}$ is tight by Inequality (\ref{claimK}), Prokhorov's theorem implies that there are a subsequence $\{n_k^1\}_{k\in \N}$ and $\nu_{t_1}^m\in \P(Y)$ such that $(p_{n_k^1})_*\nu_{t_1}^{mn_k^1}\wto \nu_{t_1}^m$ as $k\to \infty$. By diagonal argument, we can get a subsequence $\{n_k\}_{k\in \N}$ and $\{\nu_{t_l}^m\}_{l\in \N}\cup\{\nu_0^m,\nu_1^m\}=\{\nu_t^m\}_{t\in [0,1]\cap\Q}$ such that $(p_{n_k})_*\nu_t^{mn_k}\wto \nu_t^m$ for any $t\in [0,1]\cap\Q$.

Similarly, by Lemma \ref{OY18'}, Lemma \ref{OY19'} and formula (\ref{claimK}), $\{\nu_t^m\}_{m\in \N}$ is tight for any $t\in [0,1]\cap\Q$. Hence, there are a subsequence $\{m_k^1\}_{k\in \N}$ and $\nu_{t_1}\in \P(Y)$ such that $\nu_{t_1}^{m_k^1}\wto \nu_{t_1}$ as $k\to \infty$. By diagonal argument, we can get a subsequence $\{m_k\}_{k\in \N}$ and $\{\nu_{t_l}\}_{l\in \N}\cup\{\nu_0,\nu_1\}=\{\nu_t\}_{t\in [0,1]\cap\Q}$ such that $\nu_t^{m_k}\wto \nu_t$ for any $\nu_t^{m_k}\wto \nu_t$. By Lemma \ref{OY18'} and Lemma \ref{OY35'}, for any $s,t\in [0,1]\cap \Q$ and $N'\in [N, 0)$, 
\begin{align*}
S_{N', \mu_Y}(\nu_t)
&\le \liminf_{k'\to \infty}\liminf_{k\to \infty}S_{N', (p_{n_k})_*\mu_{X_{n_k}}}((p_{n_k})_*\nu_t^{m_{k'}n_k})\\
&\le \limsup_{m\to \infty}\limsup_{n\to \infty}S_{N', (p_n)_*\mu_{X_n}}((p_n)_*\nu_t^{mn})\\
&\le \sum_{i=0}^1\int_{Y\times Y}\tau_{K, N'}^{(t),i}\left(d_Y(y_0,y_1)\right)\rho_i(y_i)^{-\NN}d\pi(y_0, y_1)<\infty 
\end{align*}
and
\begin{align*}
W_2(\nu_s, \nu_t)
&\le \liminf_{k'\to \infty}\liminf_{k\to \infty}W_2(\nu_s^{m_{k'}n_k}, \nu_t^{m_{k'}n_k})\\
&\le |t-s|\limsup_{m\to \infty}\limsup_{n\to \infty}W_2(\nu_0^{mn},\nu_1^{mn})\\
&\le |t-s|W_2(\nu_0,\nu_1)
\end{align*}
hold. Thus $Y$ satisfies $\CD(K,N)$ condition.
\end{proof}
Finally, we give the proof of Claim \ref{claim}. Before that, we first prove a simple lemma.
\begin{lem}
\label{errorLip}
Let $q$ be a positive number. For any $\ep>0$ and $M>0$, there exists $C(M,\ep)>0$ such that 
\begin{equation*}
|x^q-y^q|\le C(M,\ep)|x-y|+\ep
\end{equation*}
for any $x,y\in [0,M]$.
\end{lem}
\begin{proof}
We take any $\ep>0$ and $M>0$. Because the function $x\mapsto x^q$ is uniformly continuous on $[0,M]$, there exists $\de>0$ such that any $x,y\in [0,M]$ with $|x-y|<\de$ satisfy $|x^q-y^q|<\ep$. We put 
\begin{equation*}
C(M,\ep):=\left(\frac{M^q-\ep}{\de}\right)\lor1
\end{equation*}
and take any $x,y\in [0,M]$. Then, 
\begin{equation*}
|x^q-y^q|<\ep\le C(M,\ep)|x-y|+\ep
\end{equation*}
if $|x-y|<\de$, and otherwise 
\begin{equation*}
|x^q-y^q|\le M^q\le C(M,\ep)\cdot \de+\ep\le C(M,\ep)|x-y|+\ep
\end{equation*}
holds.
\end{proof}
\begin{proof}[Proof of Claim $\ref{claim}$]
First, we consider the case where $K=0$. Then,
\begin{align*}
&\sum_{i=0}^1\int_{X_n\times X_n}\tau_{K, N'}^{(t),i}\left(d_{X_n}(x_0,x_1)\right)\rho_i^{mn}(x_i)^{-\NN}d\pi^{mn}(x_0, x_1)\\
=&(1-t)S_{N',\mu_{X_n}}(\nu_0^{mn})+ tS_{N',\mu_{X_n}}(\nu_1^{mn})
\end{align*}
and
\begin{align*}
&\sum_{i=0}^1\int_{Y\times Y}\tau_{K, N'}^{(t),i}\left(d_Y(y_0,y_1)\right)\rho_i(y_i)^{-\NN}d\pi(y_0, y_1)\\
=&(1-t)S_{N',\mu_Y}(\nu_0)+ tS_{N',\mu_Y}(\nu_1).
\end{align*}
Therefore, by (\ref{K==0}) and (\ref{K===0}), we obtain
\begin{align*}
&\limsup_{m\to \infty}\limsup_{n\to \infty}\left((1-t)S_{N',\mu_{X_n}}(\nu_0^{mn})+ tS_{N',\mu_{X_n}}(\nu_1^{mn})\right)\\
\le&\limsup_{m\to \infty}\left((1-t)S_{N',\mu_Y}(\unu_i^m)+ tS_{N',\mu_Y}(\unu_i^m)+\theta_2(m^{-1})\right)\\
=&(1-t)S_{N',\mu_Y}(\nu_0)+ tS_{N',\mu_Y}(\nu_1).
\end{align*}
Thus, Claim \ref{claim} is proved as $K=0$.

Next, we consider the case where $K\ne0$. For any sufficiently small $\eta>0$, we put the set $A_n^{\eta}$ as
\begin{equation*}
A_n^{\eta}:=
\{(x_0, x_1)\in X_n\times X_n\mid (\sgn K)\cdot(d_Y(p_n(x_0), p_n(x_1))-d_{X_n}(x_0, x_1))> \eta\}
\end{equation*}
where $\sgn K$ is the sign of $K$. Because $A_n^{\eta}$ is contained in $(X_n\times X_n)\setminus (\tX_n\times \tX_n)$ as $K>0$ for any $n$ such that $\ep_n\le \eta$, 
\begin{align*}
\limsup_{n\to \infty}\pi^{mn}(A_n^{\eta})=0
\end{align*} 
holds. On the other hand, using
\begin{align*}
&\int_{X_n\times X_n}d_{X_n}(x_0, x_1)\,d\pi^{mn}(x_0,x_1)
\le 
W_2(\nu_0^{mn}, \nu_1^{mn})
\end{align*}
and
\begin{align*}
W_2(\nu_0^m,\nu_1^m)^2&\le \int_{Y\times Y}d_Y(y_0, y_1)^2d\pi^m(y_0,y_1)\\
&\le\liminf_{n\to \infty}\int_{Y\times Y}d_Y(y_0, y_1)^2d((p_n\times p_n)_*\pi^{mn})(y_0,y_1),
\end{align*}
we get
\begin{align*}
\pi^{mn}(A_n^{\eta})=&\int_{A_n^{\eta}}d\pi^{mn}\\
<&\frac{1}{\eta^2}\int_{A_n^{\eta}}(d_{X_n}(x_0, x_1)-d_Y(p_n(x_0), p_n(x_1)))^2d\pi^{mn}(x_0,x_1)\\
\le&\frac{1}{\eta^2}\int_{A_n^{\eta}}\left(d_{X_n}(x_0, x_1)^2-d_Y(p_n(x_0), p_n(x_1))^2\right)d\pi^{mn}(x_0,x_1)\\
=&\frac{1}{\eta^2}\int_{X_n\times X_n}\left(d_{X_n}(x_0, x_1)^2-d_Y(p_n(x_0), p_n(x_1))^2\right)d\pi^{mn}(x_0,x_1)\\
&+\frac{1}{\eta^2}\int_{(X_n\times X_n)\setminus A_n^{\eta}}\left(d_Y(p_n(x_0), p_n(x_1))^2-d_{X_n}(x_0, x_1)^2\right)d\pi^{mn}(x_0,x_1)\\
\le&\frac{1}{\eta^2}\left(W_2(\nu_0^{mn}, \nu_1^{mn})^2-\int_{Y\times Y}d_Y(y_0, y_1)^2d((p_n\times p_n)_*\pi^{mn})(y_0,y_1)\right)\\
&+\frac{1}{\eta^2}\int_{X_n\times X_n}\ep_n\left(2\cdot d_{X_n}(x_0, x_1)+\ep_n\right)d\pi^{mn}(x_0,x_1)
\end{align*}
as $K<0$. Hence, 
\begin{align*}
&\limsup_{n\to \infty}\pi^{mn}(A_n^{\eta})\\
\le&\limsup_{n\to \infty}\frac{1}{\eta^2}\left(W_2(\nu_0^{mn}, \nu_1^{mn})^2-\int_{Y\times Y}d_Y(y_0, y_1)^2d((p_n\times p_n)_*\pi^{mn})(y_0,y_1)\right)\\
&+\limsup_{n\to \infty}\frac{\ep_n}{\eta^2}\left(2\cdot W_2(\nu_0^{mn}, \nu_1^{mn})+\ep_n\right)\\
\le&\frac{1}{\eta^2}\left((W_2(\unu_0^m, \unu_1^m)+\theta_2(m^{-1}))^2-\liminf_{n\to \infty}\int_{Y\times Y}d_Y(\cdot, \cdot)^2d((p_n\times p_n)_*\pi^{mn})\right)\\
\le&\frac{1}{\eta^2}\left((W_2(\unu_0^m, \unu_1^m)+\theta_2(m^{-1}))^2-W_2(\nu_0^m,\nu_1^m)^2\right)
\end{align*}
holds and if we define $\theta_3$ as $\theta_3(m^{-1}):=(W_2(\unu_0^m, \unu_1^m)+\theta_2(m^{-1}))^2-W_2(\nu_0^m,\nu_1^m)^2$, $\theta_3(m^{-1})$ converges to $0$ as $m\to \infty$ by $W_2(\unu_0^m, \unu_1^m), W_2(\nu_0^m, \nu_1^m)\to W_2(\nu_0, \nu_1)$. Thus we have
\begin{equation*}
\lim_{m\to \infty}\limsup_{n\to \infty}\pi^{mn}(A_n^{\eta})=0.
\end{equation*}

Here, we put the function $d_{\eta}:Y\times Y\to [0,\infty)$ with
\begin{equation*}
d_{\eta}(y_0,y_1):=
\begin{cases}
(d_Y(y_0,y_1)-\eta)\lor0\qquad&(K>0)\\
d_Y(y_0,y_1)+\eta\qquad&(K<0).
\end{cases}
\end{equation*}
Then, for any $(x_0, x_1)\in (X_n\times X_n)\setminus A_n^{\eta}$,
\begin{equation*}
\tau_{K,N'}^{(t),i}\left(d_{X_n}(x_0, x_1)\right)\le \tau_{K,N'}^{(t),i}\left(d_{\eta}(p_n(x_0), p_n(x_1))\right)\le C_{K,N'}^{(t),i}<\infty
\end{equation*}
holds. We put
\begin{equation*}
C_{m,n}':=\max_{j\in \J_m}\frac{\mu_Y(B_{j,m})}{\mu_{X_n}(\tB_{j,m})}
\end{equation*}
and, $C_{m,n}'$ converges to 1 as $n\to \infty$ by $\mu_{X_n}(\tB_{j,m})\to \mu_Y(B_{j,m})$. Also, we put 
\begin{equation*}
\tU_m:=\bigcup_{j\in \J_m}\tB_{j,m}
\end{equation*}
and it satisfies $\pi^{mn}(\tU_m)=1$ and 
\begin{equation*}
\urho_i^m\circ p_n=\sum_{j\in \J_m}\frac{\unu_i^m(B_{j,m})}{\mu_Y(B_{j,m})}(\1_{B_{j,m}}\circ p_n)=\sum_{j\in \J_m}\frac{\unu_i^m(B_{j,m})}{\mu_Y(B_{j,m})}\1_{\tB_{j,m}}
\end{equation*}
on $\tU_m$. Hence,
\begin{align*}
\rho_i^{mn}\le\sum_{j,k\in \J_m}w_{jk}^i\frac{1+\theta_1(m^{-1})}{\mu_{X_n}(\tB_{j,m})}\1_{\tB_{j,m}}\le(1+\theta_1(m^{-1}))\cdot C_{m,n}'\cdot \urho_i^m\circ p_n
\end{align*}
holds. By $\nu_i(B_{j,m})\le \|\rho_i\|_{\infty}\cdot \mu_Y(B_{j,m})$,
\begin{equation*}
\urho_i^m\le \frac{\|\rho_i\|_{\infty}}{\nu_i(U_m)}
\end{equation*}
holds. As we put $C_{m,n}:=\left((1+\theta_1(m^{-1}))C_{m,n}'\right)^{-\NN}$ and $B_{j,m,i}:=\proj_i^{-1}(B_{j,m})$
\begin{align*}
&\sum_{i=0}^1\int_{X_n\times X_n}\tau_{K, N'}^{(t),i}\left(d_{X_n}(x_0,x_1)\right)\rho_i^{mn}(x_i)^{-\NN}d\pi^{mn}(x_0, x_1)\\
\le&C_{m,n}\sum_{i=0}^1\int_{(X_n\times X_n)\setminus A_n^{\eta}}\tau_{K, N'}^{(t),i}\left(d_{\eta}(p_n(x_0), p_n(x_1))\right)\urho_i^m(p_n(x_i))^{-\NN}d\pi^{mn}(x_0, x_1)\\
&+C_{m,n}\sum_{i=0}^1\int_{A_n^{\eta}}C_{K, N'}^{(t),i}\cdot \left(\frac{\|\rho_i\|_{\infty}}{\nu_i(U_m)}\right)^{-\NN}d\pi^{mn}(x_0, x_1)\\
\le&C_{m,n}\sum_{i=0}^1\int_{Y\times Y}\tau_{K, N'}^{(t),i}\left(d_{\eta}(y_0, y_1)\right)\urho_i^m(y_i)^{-\NN}d((p_n\times p_n)_*\pi^{mn})(y_0, y_1)\\
&+C_{m,n}\sum_{i=0}^1C_{K, N'}^{(t),i}\cdot \left(\frac{\|\rho_i\|_{\infty}}{\nu_i(U_m)}\right)^{-\NN}\cdot\pi^{mn}(A_n^{\eta})\\
\end{align*}
holds. Then, $\tau_{K, N'}^{(t),i}\circ d_{\eta}$ is bounded continuous on $Y\times Y$, $\urho_i^m$ is a simple function, and $
\nu_i^m(\del B_{j,m})=0$ holds by $\nu_i^m\ll \mu_Y$ and $\mu_Y(\del B_{j,m})=0$. Hence, since the condition (a) of Lemma \ref{wconvint} holds, as we put the constant $C_{m,N'}:=\left(1+\theta_1(m^{-1})\right)^{-\NN}$, Lemma \ref{wconvint} gives
\begin{align}
&\limsup_{n\to \infty}\sum_{i=0}^1\int_{X_n\times X_n}\tau_{K, N'}^{(t),i}\left(d_{X_n}(x_0,x_1)\right)\rho_i^{mn}(x_i)^{-\NN}d\pi^{mn}(x_0, x_1)\notag \\
\le&C_{m,N'}\sum_{i=0}^1\int_{Y\times Y}\tau_{K, N'}^{(t),i}\left(d_{\eta}(y_0, y_1)\right)\urho_i^m(y_i)^{-\NN}d\pi^m(y_0, y_1)\notag \\
&+C_{m,N'}\cdot \frac{\theta_3(m^{-1})}{\eta^2}\sum_{i=0}^1C_{K, N'}^{(t),i}\cdot \left(\frac{\|\rho_i\|_{\infty}}{\nu_i(U_m)}\right)^{-\NN}\notag \\
\label{f3}
\le&C_{m,N'}\sum_{i=0}^1\int_{Y\times Y}\tau_{K, N'}^{(t),i}\left(d_{\eta}(y_0,y_1)\right)\rho_i(y_i)^{-\NN}d\pi^m(y_0, y_1)\\
\label{f4}
&+C_{m,N'}\sum_{i=0}^1C_{K,N'}^{(t),i}\int_Y\left|(\urho_i^m)^{-\NN}-(\rho_i)^{-\NN}\right|d\nu_i^m\\
\label{f5}
&+C_{m,N'}\cdot \frac{\theta_3(m^{-1})}{\eta^2}\sum_{i=0}^1C_{K, N'}^{(t),i}\cdot \left(\frac{\|\rho_i\|_{\infty}}{\nu_i(U_m)}\right)^{-\NN}
\end{align}
holds. Finally, we estimate these terms (\ref{f3}), (\ref{f4}) and (\ref{f5}). We take any $\ep>0$ and there exists $M>0$ such that $\rho_i, \urho_i^m\le M$ for any suffciently large $m$ and $i=0,1$ by $\nu_i(U_m)\to 1$. Let $C(M,\ep)>0$ be the constant given by Lemma \ref{errorLip} as $q=-\NN$. Then, Lemma \ref{OY38'} gives 
\begin{equation*}
\int_Y|\urho_i^m-\rho_i|d\mu_Y,\quad \int_Y|\rho_i^m-\rho_i|d\mu_Y\to 0\qquad(m\to \infty)
\end{equation*}
because $\unu_i^m$ and $\nu_i^m$ converge to $\nu_i$ weakly and the entropies of these also converge to $S_{N',\mu_Y}(\nu_i)$. Hence
\begin{align*}
&
\int_Y\left|(\urho_i^m)^{-\NN}-(\rho_i)^{-\NN}\right|d\nu_i^m\\
\le&
\int_Y\left|(\urho_i^m)^{-\NN}-(\rho_i)^{-\NN}\right|\cdot \rho_id\mu_Y
+
\int_Y\left|(\urho_i^m)^{-\NN}-(\rho_i)^{-\NN}\right|\cdot|\rho_i^m-\rho_i|d\mu_Y\\
\le&
\|\rho_i\|_{\infty}\left(C(M,\ep)\int_Y\left|\urho_i^m-\rho_i\right|d\mu_Y+\ep\right)
+
M^{-\NN}\int_Y|\rho_i^m-\rho_i|d\mu_Y\\
\end{align*}
holds about (\ref{f4}). Thus, taking limits of this inequality as $m\to \infty$ and $\ep \to 0$, we have
\begin{align*}
&\limsup_{m\to \infty}C_{m,N'}\sum_{i=0}^1C_{K,N'}^{(t),i}\int_Y\left|(\urho_i^m)^{-\NN}-(\rho_i)^{-\NN}\right|d\nu_i^m\\
\le&\lim_{\ep\to0}\sum_{i=0}^1C_{K,N'}^{(t),i}\cdot \|\rho_i\|_{\infty}\cdot \ep=0.
\end{align*}
Next, about (\ref{f3}), since the condition (c) of Lemma \ref{wconvint} holds, Lemma \ref{wconvint} implies
\begin{align*}
&\limsup_{m\to \infty}C_{m,N'}\sum_{i=0}^1\int_{Y\times Y}\tau_{K, N'}^{(t),i}\left(d_{\eta}(y_0,y_1)\right)\rho_i(y_i)^{-\NN}d\pi^m(y_0, y_1)\\
\le&\sum_{i=0}^1\int_{Y\times Y}\tau_{K, N'}^{(t),i}\left(d_{\eta}(y_0,y_1)\right)\rho_i(y_i)^{-\NN}d\pi(y_0, y_1).
\end{align*}

Finally, (\ref{f5}) satisfies
\begin{align*}
\limsup_{m\to \infty}C_{m,N'}\cdot \frac{\theta_3(m^{-1})}{\eta^2}\sum_{i=0}^1C_{K, N'}^{(t),i}\cdot \left(\frac{\|\rho_i\|_{\infty}}{\nu_i(U_m)}\right)^{-\NN}=0.
\end{align*}
Thus, we obtain
\begin{align*}
&\limsup_{m\to \infty}\limsup_{n\to \infty}\sum_{i=0}^1\int_{X_n\times X_n}\tau_{K, N'}^{(t),i}\left(d_{X_n}(x_0,x_1)\right)\rho_i^{mn}(x_i)^{-\NN}d\pi^{mn}(x_0, x_1)\\
\le&\sum_{i=0}^1\int_{Y\times Y}\tau_{K, N'}^{(t),i}\left(d_{\eta}(y_0,y_1)\right)\rho_i(y_i)^{-\NN}d\pi(y_0, y_1).
\end{align*}
Since $d_{\eta}(y_0,y_1)$ converges to $d_Y(y_0,y_1)$ as $\eta\to 0$, by Lebesgue's convergence theorem, we obtain
\begin{align*}
&\limsup_{m\to \infty}\limsup_{n\to \infty}\sum_{i=0}^1\int_{X_n\times X_n}\tau_{K, N'}^{(t),i}\left(d_{X_n}(x_0,x_1)\right)\rho_i^{mn}(x_i)^{-\NN}d\pi^{mn}(x_0, x_1)\\
\le&\sum_{i=0}^1\int_{Y\times Y}\tau_{K, N'}^{(t),i}\left(d_Y(y_0,y_1)\right)\rho_i(y_i)^{-\NN}d\pi(y_0, y_1)
\end{align*}
i.e. we get Claim \ref{claim}.
\end{proof}
\begin{rem}
As noted in Remark \ref{rem}, if $K\ge 0$, Main theorem \ref{mainthm1} for $\CD^*(K,N)$ spaces can also be proved by replacing $C_{K,N'}^{(t),i}$ by $\sup_{\theta\in[0,\infty)}\si_{K/N'}^{(t),i}(\theta)<\infty$.
\end{rem}
Theorem \ref{mainthm1} can imply the following theorem.
\begin{thm}[cf. {\cite[Corollary 1.4]{Funano-Shioya}}]
Let $N$ be a negative number. If a sequence of compact Riemannian manifolds $\{M_n\}_{n\in \N}$ satisfies two following conditions,
\begin{enumerate}
\item[(d)] $M_n$ is a $\CD(0,N)$ space.
\item[(e)] $\displaystyle \lim_{n\to \infty}\la_k(M_n)=\infty$ for some $k\in \N$.
\end{enumerate}
Then, $\{M_n\}_{n\in \N}$ is a L\'{e}vy family.
\end{thm}
\begin{proof}
By the condition (e) and Remark \ref{eigenLevy}, $\{M_n\}_{n\in \N}$ is a $k$-L\'{e}vy family. Hence, we have only one of (1) or (2) of Proposition \ref{kLevy}. If (2) holds, there exist a subsequence $\{n_i\}_{i\in \N}$, a sequence $\{t_i\}_{i\in \N}\subset (0,1]$ and an mm-space $Y$ with $|Y|\in \{2,\ldots,k\}$ such that $t_iM_{n_i}\concto Y$ holds. Then, $t_iM_{n_i}$ is a $\CD(0,N)$ space by the condition (d) and Proposition \ref{CDscale}. Hence, Main Theorem \ref{mainthm1} implies that $Y$ is $\CD(0,N)$. However, since any finite mm-space except $*$ is not a $\CD(0,N)$ space, it is a contradiction. Thus, we obtain (1) of Proposition \ref{kLevy}, i.e., $\{M_n\}_{n\in \N}$ is a L\'{e}vy family.
\end{proof}
\section{Proof of Main Theorem \ref{mainthm2}}
Next, we give the proof of Main Theorem \ref{mainthm2}. Before we do so, we define a function and prove its properties.
\begin{dfn}
Let $a>0$. We define a function $F_a:\R \to (0,\infty)$ defined by 
\begin{equation*}
F_a(x):=a^{-1}\log\left(e^{ax}+e^{-ax}\right).
\end{equation*}
\end{dfn}
This function $F_a$ satisfies 
\begin{equation*}
|x|<F_a(x)\le |x|+a^{-1}\log2
\end{equation*}
for any $a>0$ and $x\in \R$ and this implies $F_a(x)\to |x|$ as $a\to \infty$.
\begin{lem}
\label{LSEcvx}
Let $a>0$, $\beta\ge0$ and let $U$ be a nonempty open set in $\R$. If $f\in C^{\infty}(U)$ and $x\in U$ satisfy $f''(x)+\beta f(x)\ge0$ and $f(x)\ge 0$, then $F_a(f(x))''+\beta F_a(f(x))\ge 0$ holds.
\end{lem}
\begin{proof}
By differentiating $F_a(f)=a^{-1}\log\left(e^{af}+e^{-af}\right)$ twice, we get
\begin{align*}
F_a(f)'&=\frac{f'e^{af}-f'e^{-af}}{e^{af}+e^{-af}}\\
F_a(f)''&=\frac{f''e^{af}+a\left(f'\right)^2e^{af}-f''e^{-af}+a\left(-f'\right)^2e^{-af}}{e^{af}+e^{-af}}-\frac{a\left(f'e^{af}-f'e^{-af}\right)^2}{\left(e^{af}+e^{-af}\right)^2}\\
&=\frac{f''e^{af}-f''e^{-af}}{e^{af}+e^{-af}}+\frac{4a\left(f'\right)^2}{\left(e^{af}+e^{-af}\right)^2}.
\end{align*}
Thus, $f''(x)+\beta f(x)\ge0$ and $f(x)\ge 0$ imply 
\begin{align*}
F_a(f(x))''\ge f''(x)\frac{e^{af}-e^{-af(x)}}{e^{af(x)}+e^{-af(x)}}
\ge -\be f(x)\cdot 1\ge -\be F_a(f(x)),
\end{align*}
i.e., we have $F_a(f(x))''+\beta F_a(f(x))\ge 0$.
\end{proof}
\begin{proof}[Proof of Main Theorem $\ref{mainthm2}$]
We put $r:=D/\pi$, $N':=N-1$, $\ka:=K/N'$ and $Y:=\{(-1,0),(1,0)\}\subset \Ss^1$. 
We define a sequence of smooth functions $f_n:\Ss^1\to \R$ by $f_n(x,y):=-N'\log(a_n \cdot F_n(y))$, where $a_n>0$ is the normalising constant, i.e., 
\begin{align*}
a_n:=\left(\int_{\Ss^1}F_n(y)^{N'}d\vol_g(x,y)\right)^{-\NN}.
\end{align*}
Then, we have $e^{-f_n}\vol_g\in \P(\Ss^1)$. Indeed, 
\begin{align*}
\left(e^{-f_n}\vol_g\right)(\Ss^1)&=\int_{\Ss^1}e^{-f_n}d\vol_g\\
&=\int_{\Ss^1}\left(a_n\cdot F_n(y)\right)^{N'}d\vol_g(x,y)\\
&=a_n^{N'}\cdot a_n^{-N'}=1.
\end{align*}
Hence, $(\Ss^1, g, e^{-f_n}\vol_g)$ is a $1$-dimensional weighted Riemannian manifold, and $\{(\Ss^1, d_g, e^{-f_n}\vol_g)\}_{n\in \N}$ is a sequence of mm-spaces whose diameter is $D$. We will prove the following two conditions
\begin{enumerate}
\item[(C)] $(\Ss^1, d_g, e^{-f_n}\vol_g)$ satisfies the $\CD(K,N)$ condition.
\item[(D)] $(\Ss^1, d_g, e^{-f_n}\vol_g)$ $\Box$-converges to $(Y, d_g, 2^{-1}(\de_{(-1,0)}+\de_{(1,0)}))$.
\end{enumerate}

First, to show the condition (C), it suffices to prove that $f_n$ is a $(K,N')$-convex function by Proposition \ref{mfdCD} since $(\Ss^1, d_g, \vol_g)$ satisfies $\Ric_1\ge0$. Since the function $\ph:\R\to \Ss^1; \ph(t):=(\cos(r^{-1}t),\sin(r^{-1}t))$ is an isometric embedding when restricted on any open interval of length less than or equal to $2D$. Thus,  
in order to prove $\Hess\left(e^{-f_n/N'}\right)\ge -\ka g$, we prove $(a_n\cdot F_n(f))''+\ka (a_n\cdot F_n(f))\ge 0$ on $(0,2D)$ where $f(t):=\sin(r^{-1}t)$. If $f(t)\ge0$ i.e. $t\in (0, D]$, we have $f''(t)+\ka f(t)=(\ka-r^{-2})f(t)\ge 0$ by $r=D/\pi\ge \ka^{-1/2}$. Lemma \ref{LSEcvx} implies 
\begin{equation*}
(a_n\cdot F_n(f(t)))''+\ka (a_n\cdot F_n(f(t)))=a_n(F_n(f(t))''+F_n(f(t)))\ge 0.
\end{equation*}
If $f(t)<0$ i.e. $t\in (D, 2D)$, $h=-f$ satisfies $h(t)\ge 0$ and $h''(t)+\ka h(t)\ge 0$. Hence, $F_n(h)=F_n(f)$ and Lemma \ref{LSEcvx} implies $(a_n\cdot F_n(f(t)))''+\ka a_n\cdot F_n(f(t))\ge 0$. Thus, the condition (C) holds.

Next, we will show the condition (D). Since $(Y, d_g, 2^{-1}(\de_{(-1,0)}+\de_{(1,0)}))$ is mm-isomorphic to $(\Ss^1, d_g, 2^{-1}(\de_{(-1,0)}+\de_{(1,0)}))$, it suffices to prove $\mu_{X_n}:=e^{-f_n}\vol_g\wto 2^{-1}(\de_{(-1,0)}+\de_{(1,0)})=:\mu_Y$ by Proposition \ref{Boxwconv}. For that purpose, it is sufficient to show that for any $\ep\in (0,D/2)$,
\begin{equation*}
\lim_{n\to \infty}\mu_{X_n}(\Ss^1\setminus N_{\ep}(Y))=0
\end{equation*}
because $f_n(x,y)$ is independent of $x$ and an even function for $y$.
First, we prove $a_n\to \infty$ as $n\to \infty$ beforehand. For any $\ep \in (0, D/2)$, 
\begin{align*}
\liminf_{n\to \infty}a_n&\ge\liminf_{n\to \infty}\left(\int_{\ph([D/2,D-\ep])}F_n\left(y\right)^{N'}d\vol_g(x,y)\right)^{-\NN}\\
&\ge\left(\int_{D/2}^{D-\ep}\left(\lim_{n\to\infty}F_n(f(t))\right)^{N'}dt\right)^{-\NN}\\
&=\left(\int_{D/2}^{D-\ep}f(t)^{N'}dt\right)^{-\NN}
\end{align*}
and $0<f(t)\le r^{-1}(D-t)$ as $t\in [D/2,D-\ep]$ imply
\begin{align*}
\int_{D/2}^{D-\ep}f(t)^{N'}dt&\ge \int_{D/2}^{D-\ep}\left(r^{-1}(D-t)\right)^{N'}dt\\
&=\left[-\frac{r^{-N'}}{N}(D-t)^N\right]_{D/2}^{D-\ep}\\
&=-\frac{r^{-N'}}{N}\left(\ep^N-(D/2)^N\right).
\end{align*}
Since this value diverges to infinity as $\ep \to 0$, $a_n\to \infty$ holds. We take any $\ep \in (0, D/2)$. For any $t\in [\ep,D-\ep]\cup[D+\ep,2D-\ep]$, $F_n(f(t))\ge |f(t)|\ge f(\ep)>0$ and $\Ss^1\setminus N_{\ep}(Y)=\ph([\ep,D-\ep]\cup[D+\ep,2D-\ep])$ imply
\begin{align*}
\mu_{X_n}(\Ss^1\setminus N_{\ep}(Y))=&\int_{\Ss^1\setminus N_{\ep}(Y)}\left(a_n\cdot F_n(y)\right)^{N'}d\vol_g(x,y)\\
=&a_n^{N'}\left(\int_{\ep}^{D-\ep}F_n(f(t))^{N'}dt+\int_{D+\ep}^{2D-\ep}F_n(f(t))^{N'}dt\right)\\
\le&a_n^{N'}\cdot f(\ep)^{N'}\cdot(2D-4\ep)\\
\to& 0\qquad(n\to \infty) 
\end{align*}
Thus, $X_n$ is a $\CD(K,N)$ space for any $n\in \N$ and $\Box$-converges to $Y$.
\end{proof}
\section{Appendix}
In this appendix, we give our results for metric measure spaces that are not mm-spaces. A triple $(X,d_X,\mu_X)$ is a metric measure space if $(X,d_X)$ is a complete separable metric space and $\mu_X$ is a locally finite non-zero Borel measure, i.e., for any $x\in X$, there exists $r>0$ such that 
\begin{equation*}
\mu_X\left(B_X(x,r)\right)<\infty.
\end{equation*}
We can also consider the $\CD(K,N)$ condition for metric measure spaces defined in this way.
\begin{dfn}[Volume growth condition]
A metric measure space $(X,d_X,\mu_X)$ satisfies the \emph{volume growth condition} if there exist $C>0$ and $x_0\in X$ such that
\begin{equation*}
\int_Xe^{-Cd_X(x,x_0)^2}d\mu_X(x)<\infty.
\end{equation*}
\end{dfn}
According to Theorems 4.24 and 4.26 in \cite{Sturm1}, every $\CD(K,\infty)$ space for some $K\in \R$ satisfies the volume growth condition, and its measure is finite if $K>0$. However, for $N<0$, there are $\CD(K,N)$ spaces that do not satisfy the volume growth condition even if $K>0$.
\begin{ex}
For $K>0$ and $N<0$, we put a function $f:\R\to \R$ with
\begin{equation*}
f(x):=-(N-1)\sqrt{\frac{1}{4}-\frac{K}{N-1}}\sinh x.
\end{equation*}
Then, $(\R,|\cdot|,e^{-f}\LL^1)$ is a $\CD(K,N)$ space and does not satisfy the volume growth condition.
\end{ex}
\begin{proof}
First, we prove that $(\R,|\cdot|,e^{-f}\LL^1)$ does not satisfy the volume growth condition. For any $C>0$ and $x_0\in \R$, the function $C|\cdot-x_0|^2+f$ is not bounded from below. Thus, 
\begin{equation*}
\int_{\mathbb{R}}e^{-C|x-x_0|^2}d(e^{-f}\LL^1)=\int_{-\infty}^{\infty}e^{-(C|x-x_0|^2+f(x))}dx=\infty.
\end{equation*}

Next, we prove that $(\R,|\cdot|,e^{-f}\LL^1)$ is a $\CD(K,N)$ space. By Proposition \ref{mfdCD}, it is sufficient to prove that $f$ is a $(K,N-1)$-convex function. We put a constant $a:=\sqrt{\frac{1}{4}-\frac{K}{N-1}}$ and a function $g:=e^{-f/(N-1)}$. Then, 
\begin{align*}
g(x)&=e^{a\sinh x}\\
g'(x)&=ae^{a\sinh x}\cosh x\\
g''(x)&=ae^{a\sinh x}\sinh x+a^2e^{a\sinh x}(\cosh x)^2
\\
&=\left(a^2\left(\sinh x+(2a)^{-1}\right)^2-4^{-1}+a^2\right)e^{a\sinh x}\\
&\ge \left(a^2-4^{-1}\right)e^{a\sinh x}=-\frac{K}{N-1}g(x)
\end{align*}
imply that $f$ is a $(K,N-1)$-convex function.
\end{proof}
In general, it is known that the $\CD(K,N)$ space (and $\CD^*(K,N)$ space) becomes smaller when $K$ is large. For instance, if $X$ is a $\CD(K,N)$ space for some $K>0$ and $N\in (1,\infty)$, then
\begin{align*}
\diam \supp \mu_X\le \pi\sqrt{\frac{N-1}{K}}
\end{align*}
holds ({\cite[Corollary 2.6]{Sturm2}}). But this is not always the case when $N<0$, as in the example given above. However, when ``$K=\infty$", the result is the same as when $N>1$.
\begin{thm}
\label{K=infty}
Let $(X,d_X, \mu_X)$ be a metric measure space. If there exists $N<0$ such that $(X,d_X, \mu_X)$ is a $\CD^*(K,N)$ space for any $K>0$, then $\supp \mu_X$ consists of one point.
\end{thm}
We state a proposition necessary to prove this theorem.
\begin{prop}[{\cite[Theorem 4.8]{Ohta}}]
\label{BMine}
Let $(X,d_X,\mu_X)$ be a $\CD^*(K,N)$ space for some $K\in \R$ and $N<0$. For any $A_0, A_1\in \B_X$ with $\diam(A_0\cup A_1)<\pi\sqrt{N/K}$ as $K<0$ and for any $t\in [0,1]$, we have 
\begin{equation}
\label{BMinequality}
\mu_X(A_t)^{\frac{1}{N}}\le \sum_{i=0}^1\sup_{x\in A_0, y\in A_1}\si_{K/N}^{(t),i}\left(d_X(x,y)\right)\mu_X(A_i)^{\frac{1}{N}}
\end{equation}
where $A_t$ is defined as
\begin{equation*}
A_t:=\{\ga_t\in X\mid \text{$\ga$ is a geodesic in $X$ with $\ga_0\in A_0$ and $\ga_1\in A_1$}\}.
\end{equation*}
\end{prop}
We call this inequality (\ref{BMinequality}) Brunn-Minkowski inequality.
\begin{proof}[Proof of Theorem $\ref{K=infty}$]
We assume that there exist two points $x_0,x_1\in \supp \mu_X$ with $x_0\ne x_1$. We take $r\in (0,d_X(x_0, x_1)/2)$ with $\mu_X(B_X(x_0, 3r))<\infty$ and $\mu_X(B_X(x_1, r))<\infty$ (locally finiteness of $\mu_X$ implies an existence of such $r$). We put $\ep:=d_X(x_0, x_1)$ and $A_i:=B_X(x_i,r)$ for $i=0,1$. Then, since $(X,d_X, \mu_X)$ is a $\CD^*(K,N)$ space, Proposition \ref{BMine} implies 
\begin{equation*}
\mu_X(A_t)^{\frac{1}{N}}\le \sum_{i=0}^1\sup_{x\in A_0, y\in A_1}\si_{K/N}^{(t),i}\left(d_X(x,y)\right)\mu_X(A_i)^{\frac{1}{N}}
\end{equation*}
for any $t\in [0,1]$ and $K>0$. If $t\in (0,1)$, we have that $\si_{K/N}^{(t),i}(\theta)$ is decreasing for $\theta$ and converges to $0$ as $\theta\to\infty$ by $K>0$.
Furthermore, for any $x\in A_0$, $y\in A_1$,
\begin{equation*}
d_X(x,y)\ge d_X(x_0,x_1)-d_X(x_0,x)-d_X(y,x_1)>\ep -2r>0
\end{equation*}
holds. Thus $\si_{K/N}^{(t),i}(\theta)=\si_{-1}^{(t),i}(\sqrt{K/(-N)}\theta)$ implies
\begin{align*}
\mu_X(A_t)^{\frac{1}{N}}&\le \sum_{i=0}^1\si_{-1}^{(t),i}\left(\sqrt{K/(-N)}\cdot(\ep-2r)\right)\mu_X(A_i)^{1/N}.
\end{align*}
Because $A_t$ is independent of $K$ and this inequality holds for any $K>0$, we obtain $\mu_X(A_t)^{1/N}\le0$ i.e. $\mu_X(A_t)=\infty$ by taking $K\to \infty$. 
But, if we put $n\ge 2$ with $\ep/(n+1)\le r< \ep/n$ and $t:=(n+1)^{-1}\in (0,1/2)$, and take any $y\in A_t$, there exists a geodesic $\ga$ in $X$ such that $\ga_0\in A_0$, $\ga_1\in A_1$ and $y=\ga_t$. Then,
\begin{align*}
d_X(x_0,y)&\le d_X(x_0,\ga_0)+d_X(\ga_0,\ga_t)\\
&< r+td_X(\ga_0,\ga_1)\\
&< r+t(d_X(\ga_0,x_0)+d_X(x_0,x_1)+d_X(x_1,\ga_1))\\
&< r+t(\ep+2r)
< 3r
\end{align*}
implies $A_t\subset B_X(x_0, 3r)$, and it contradicts $\mu_X(B_X(x_0, 3r))<\infty$. Thus $\supp \mu_X$ consists of one point.
\end{proof}


\begin{bibdiv}
\begin{biblist}    

\bib{Funano-Shioya}{article}{
   author={Funano, Kei},
   author={Shioya, Takashi},
   title={Concentration, Ricci curvature, and eigenvalues of Laplacian},
   journal={Geom. Funct. Anal.},
   volume={23},
   date={2013},
   number={3},
   pages={888--936},
   issn={1016-443X},
}
\bib{Gigli-Mondino-Savare}{article}{
   author={Gigli, Nicola},
   author={Mondino, Andrea},
   author={Savar\'{e}, Giuseppe},
   title={Convergence of pointed non-compact metric measure spaces and
   stability of Ricci curvature bounds and heat flows},
   journal={Proc. Lond. Math. Soc. (3)},
   volume={111},
   date={2015},
   number={5},
   pages={1071--1129},
   issn={0024-6115},
}
\bib{Gromov}{book}{
   author={Gromov, Misha},
   title={Metric structures for Riemannian and non-Riemannian spaces},
   series={Modern Birkh\"auser Classics},
   edition={Reprint of the 2001 English edition},
   note={Based on the 1981 French original;
   With appendices by M. Katz, P. Pansu and S. Semmes;
   Translated from the French by Sean Michael Bates},
   publisher={Birkh\"auser Boston, Inc., Boston, MA},
   date={2007},
   pages={xx+585},
   isbn={978-0-8176-4582-3},
   isbn={0-8176-4582-9},
}
\bib{Kazukawa-Ozawa-Suzuki}{article}{
   author={Kazukawa, Daisuke},
   author={Ozawa, Ryunosuke},
   author={Suzuki, Norihiko},
   title={Stabilities of rough curvature dimension condition},
   journal={J. Math. Soc. Japan},
   volume={72},
   date={2020},
   number={2},
   pages={541--567},
   issn={0025-5645},
}
\bib{Lott-Villani}{article}{
   author={Lott, John},
   author={Villani, C\'{e}dric},
   title={Ricci curvature for metric-measure spaces via optimal transport},
   journal={Ann. of Math. (2)},
   volume={169},
   date={2009},
   number={3},
   pages={903--991},
   issn={0003-486X},
}

\bib{Magnabosco-Rigoni-Sosa}{article}{
   author={Magnabosco, Mattia},
   author={Rigoni, Chiara},
   author={Sosa, Gerardo},
   title={Convergence of metric measure spaces satisfying the CD condition for negative values of the dimension parameter},
   note={preprint (2021), arXiv:2104.03588},
}
\bib{Milman}{article}{
   author={Milman, Emanuel},
   title={Harmonic measures on the sphere via curvature-dimension},
   journal={Ann. Fac. Sci. Toulouse Math. (6)},
   volume={26},
   date={2017},
   number={2},
   pages={437--449},
   issn={0240-2963},
}
\bib{Ohta}{article}{
   author={Ohta, Shin-ichi},
   title={$(K,N)$-convexity and the curvature-dimension condition for
   negative $N$},
   journal={J. Geom. Anal.},
   volume={26},
   date={2016},
   number={3},
   pages={2067--2096},
   issn={1050-6926},
}
\bib{Ozawa-Yokota}{article}{
   author={Ozawa, Ryunosuke},
   author={Yokota, Takumi},
   title={Stability of RCD condition under concentration topology},
   journal={Calc. Var. Partial Differential Equations},
   volume={58},
   date={2019},
   number={4},
   pages={Paper No. 151, 30},
   issn={0944-2669},
}
\bib{ShioyaMMG}{book}{
   author={Shioya, Takashi},
   title={Metric measure geometry. Gromov's theory of convergence and concentration of metrics and
   measures},
   series={IRMA Lectures in Mathematics and Theoretical Physics},
   volume={25},
   publisher={EMS Publishing House, Z\"{u}rich},
   date={2016},
   pages={xi+182},
   isbn={978-3-03719-158-3},
}
\bib{Sturm1}{article}{
   author={Sturm, Karl-Theodor},
   title={On the geometry of metric measure spaces. I},
   journal={Acta Math.},
   volume={196},
   date={2006},
   number={1},
   pages={65--131},
   issn={0001-5962},
}
\bib{Sturm2}{article}{
   author={Sturm, Karl-Theodor},
   title={On the geometry of metric measure spaces. II},
   journal={Acta Math.},
   volume={196},
   date={2006},
   number={1},
   pages={133--177},
   issn={0001-5962},
}
\bib{Villani}{book}{
   author={Villani, C\'{e}dric},
   title={Optimal transport, old and new},
   series={Grundlehren der mathematischen Wissenschaften [Fundamental
   Principles of Mathematical Sciences]},
   volume={338},
   publisher={Springer-Verlag, Berlin},
   date={2009},
   pages={xxii+973},
   isbn={978-3-540-71049-3},
}

\end{biblist}
\end{bibdiv}
\end{document}